\documentclass[preprint,3p,authoryear,a4paper]{elsarticle}
\usepackage{graphicx}
\usepackage{xcolor}
\usepackage{dsfont}
\usepackage{amsmath}
\usepackage{amssymb}
\usepackage{amsfonts}
\usepackage{amsbsy}
\usepackage{amsthm}
\usepackage{array}
\usepackage{booktabs} 
\usepackage{verbatim} 
\usepackage[english]{babel}
\usepackage{float}
\usepackage{epsfig}
\usepackage{pdfsync}
\usepackage{verbatim}
\usepackage{mathrsfs}
\usepackage{subcaption}

\newcolumntype{C}[1]{>{\centering\let\newline\\\arraybackslash\hspace{0pt}}m{#1}}
\renewcommand\appendix{\par
\setcounter{section}{0}%
\setcounter{subsection}{0}%
\setcounter{table}{0}
\setcounter{table}{0}
\setcounter{figure}{0}
\gdef\thetable{\Alph{table}}
\gdef\thefigure{\Alph{figure}}
\gdef\thesection{\Alph{section}}
\setcounter{section}{0}}

\newtheorem{theorem}{Theorem}[section]
\newtheorem{lemma}[theorem]{Lemma}

\newtheorem{corollary}[theorem]{Corollary}
\newtheorem{definition}[theorem]{Definition}

\newcommand{\corr}[1]{\textcolor{black}{#1}}
\newcommand{\corrr}[1]{\textcolor{red}{#1}}

\newcommand{\Ex}{{\mathbb{E}}_x}

\newcommand{\pp}{^{\prime\prime}}

\newcommand{\phiqr}{\phi_{\gamma+\delta}}

\newcommand{\Wq}{W_{\delta}}
\newcommand{\Zq}{Z_{\delta}}
\newcommand{\Zqr}{Z_{\gamma,\delta}}

\newcommand{\Wqr}{W_{\gamma+\delta}}
\newcommand{\Wb}{\overline{W}_{\gamma+\delta}}
\newcommand{\Wbb}{\overline{\overline{W}}_{\gamma+\delta}}
\newcommand{\Wa}{W_{\gamma,\delta,a}}
\newcommand{\WB}{W_{\gamma,\delta,b}}

\newcommand{\Ma}{\mathscr{M}^{(\gamma,\delta)}_a}
\newcommand{\T}{{T_{E_B}}}

\newcommand{\pis}{{\pi_{b_u,b_l}^{\kappa,s}}}
\newcommand{\pio}{{\pi_{b_u,b_l}^{\kappa,*}}}

\newcommand{\Vs}{{V_s}}
\newcommand{\Vo}{{V_*}}
\newcommand{\Voo}{V_\kappa^\prime(0;\pi^{\kappa,s}_{b_u,0})}

\begin{document}

\begin{frontmatter}

\title{Optimal periodic dividend strategies for spectrally negative L\'evy processes with fixed transaction costs}

\author[UMelb]{Benjamin Avanzi}
\ead{b.avanzi@unimelb.edu.au}

\author[UNSW]{Hayden Lau\corref{cor}}
\ead{kawai.lau@unsw.edu.au}

\author[UNSW]{Bernard Wong}
\ead{bernard.wong@unsw.edu.au}

\cortext[cor]{Corresponding author.}

\address[UMelb]{Centre for Actuarial Studies, Department of Economics \\ University of Melbourne VIC 3010, Australia}
\address[UNSW]{School of Risk and Actuarial Studies, UNSW Australia Business School\\ UNSW Sydney NSW 2052, Australia}

\begin{abstract}

Maximising dividends is one classical stability criterion in actuarial risk theory. Motivated by the fact that dividends are paid periodically in real life, \emph{periodic} dividend strategies were  recently introduced \citep*{AlGeSh11}. In this paper, we incorporate fixed transaction costs into the model and study the optimal periodic dividend strategy with fixed transaction costs for spectrally negative L\'evy processes. 

The value function of a periodic $(b_u,b_l)$ strategy is calculated by means of exiting identities and It\^o's excusion when the surplus process is of unbounded variation. 
We show that a sufficient condition for optimality is that the L\'evy measure admits a density which is completely monotonic. Under such assumptions, a periodic $(b_u,b_l)$ strategy is confirmed to be optimal.

Results are illustrated.

\end{abstract}

\begin{keyword}
Optimal periodic dividends \sep SNLP \sep Fixed transaction costs

JEL codes: 
C44 \sep 
C61 \sep 
G24 \sep 
G32 \sep 
G35 


\end{keyword}

\end{frontmatter}

\newtheorem{remark}{Remark}[section]
\numberwithin{equation}{section}

\section{Introduction} \label{S_intro}

The first to study the now so-called ``stability problem'' were \citet*{Lun09,Cra30}, with the traditional stability criterion being the probability of ruin \citep*{Buh70,Ger72}. A major criticism of this set-up is that companies do not let their surplus grow to infinity (as the probability of ruin criterion suggests they should), and they do distribute profits to their beneficiaries eventually. This means that the ruin probability calculated does not actually represent the probability of ruin of the company (even with infinite horizon) - the calculations are flawed \citep*[as first argued by][]{deF57}. In addition, making decisions based on probability of ruin does not capture the risk and reward trade-off which companies typically face. Because of this, Bruno \citet*{deF57} first introduced an alternative formulation where distribution of surplus, or `dividends', is allowed and the stability criterion is the maximised expected present value of dividends paid until ruin. This formulation is arguably more balanced, as neither paying too much nor too little will maximise the dividends. A strategy that maximises dividends is called an optimal strategy and the form of an optimal strategy is of particular interest. Since then, the optimal dividend problem for an insurance company has been studied intensively \citep*[see, for instance,][]{Ava09,AlTh09}. 

While this literature does not belong to corporate finance, and does not really mean to (directly) inform companies how one should pay dividends in real life, the qualitative results we obtain from the modelling can only be improved by making the dividend strategies more realistic \citep*[see][for a formal discussion of what `realistic' means in this context]{AvTuWo16c}. 

In this spirit, periodic dividend strategies were introduced by \cite*{AlChTh11a} and have caught some recent attention as they capture the periodicity of dividend payments in real life. A periodic dividend strategy refers to the scenario when dividends can only be paid at some ``specified'' times. One motivation of this setting is that companies typically distribute (part of) their surplus to shareholders (as dividends or share buy-backs, for instance) at specific times in a year. Unfortunately, paying dividends at deterministic times introduces technical difficulties as one needs to keep track of the time until the next payment time. However, an \textit{Erlang}-$n$ random variable can be used as an approximation to a deterministic constant. This technique was first used in ruin theory by \citet*{AsAvUs02} to approximate the probability of ruin in finite time. The same technique was subsequently introduced in the dividend setting by \citet*{AlChTh11a}. They considered the case when the dividend payment times (also called ``decision times'' as the dividends are ``decided'' and paid instantly at those times) are random variables and the solvency of the company is also considered at that same period. This means that a negative surplus is possible as long as it reverts to a non-negative value at the next observation time. This is related to the concept of Parisian (soft) ruin, where it is argued that companies do not go bankrupt instantaneously and may be able to recover before bankruptcy. \citet*{AvChWoWo13} studied periodic barrier strategies with continuous monitoring of solvency, that is, when ruin happens as soon as the surplus hits 0 (the assumption in this paper). 

In this paper, we determine the optimal periodic dividend strategy under spectrally negative L\'evy process, in presence of \emph{fixed} transaction costs (see Remark \ref{Remark1.1} below). Here, transaction costs refer to the costs of transferring the surplus of the company to the shareholders. This includes both explicit components (e.g., tax and administrative costs), but also potentially implicit components (e.g., opportunity costs, penalty if it is undesirable to pay too often). The studies of optimal dividend strategies under fixed transaction costs have been done in `continuous' decision making models \citep*[`continuous' here is mentioned as opposed to `periodic'; see also][for a discussion of the interaction between periodic and continuous dividend decisions]{AvTuWo16}; see for example \citet*[in the Brownian model]{JeSh95}, \citet*[\corr{for spectrally negative L\'evy processes}]{Loe08a}, \citet*[in the dual model]{BaKyYa13} and \citet*[with taxes]{ChYuWa20}. Although inspired by the Erlang-$n$ technique, we only consider the case when $n=1$ in this paper to enable neat expressions and formula, see \citet*{AlIvZh16}. This means that a dividend decision time is activated when a Poisson process jumps.

\begin{remark}\label{Remark1.1}
	Transaction costs are typically comprised of two components, proportional costs and fixed costs. For example, if a dividend amount of $\xi$ is paid, the cost is $\rho\xi+\kappa$, where $\corr{\kappa}\geq 0$ and $1>\rho\geq 0$. Note that the complexity of the problem does not increase with the presence of proportional cost $\rho$ (e.g. tax) because it can be removed by scaling the risk metric (or currency) with a ratio of $(1-\rho)^{-1}$ and considering another fixed cost $\kappa'$. Hence, we can assume without loss of generality that there are no proportional costs. 
\end{remark}

L\'evy processes \citep*{Ber98} encompass a wide class of models present in the literature, including the Cram\'er-Lundberg \citep*{Ger69}, Brownian \citep*{Ger70}, and dual models \citep*{AvGeSh07}. When there are no positive jumps---models are then referred to as `spectrally negative'---fluctuation theory takes the nicest form and various quantities can be expressed explicitly in terms of scale functions $W_q$ \citep*[see][p. 239, for a remark regarding their historical development]{Kyp14}.

Thanks to the recent development in the theories regarding spectrally negative L\'evy processes, the Cram\'er-Lundberg model is often extended and studied as a spectrally negative L\'evy process. \corr{Classical examples of spectrally negative L\'evy processes include Brownian motion with drift, Cram\'er-Lundberg risk processes and $\alpha$-stable processes with stability parameter $\alpha\in(1,2)$. More recently, a new family of examples known as Gaussian Tempered Stable Convolution class has been derived in \citet*{HuFrKy08} \citep*[See also][]{Loe08a,Loe09,Loe09b,AvPaPi07,WaZh18,WaXu20,XuWaGa20}.}

For spectrally negative L\'evy processes (``SNLP''), it is known that in general barrier types of strategy are not necessarily optimal \citep*[e.g.,][]{Ger69,AzMu05}. It is also observed in \citet*{AvPaPi07} that the shape of the scale function $W_q$ plays an important role in the optimal dividend problem for spectrally negative L\'evy process. In particular, one sufficient condition for the barrier type of strategy to be optimal is that the L\'evy measure has a completely monotonic density \corr{(e.g. Cram\'er-Lundberg under certain conditions \corrr{such as with hyper exponentially distributed jumps,}  or (one-sided) tempered stable processes as discussed above); see} \citet*{Loe08}. Under such assumption, \citet*{NoPeYaYa17} recently proved that a periodic barrier strategy is optimal when the surplus is a spectrally negative L\'evy process. Extending \citet*{Loe08} and \citet*{NoPeYaYa17}, we show that a periodic $(b_u,b_l)$ strategy is also optimal under SNLP with the same assumption on the L\'evy measure, when fixed transaction costs on dividends are present. \corr{Sometimes, forced capital injection is used to replace the sufficient condition on the L\'evy process, see e.g. \citet*{KuSc08}}.

The paper is organised as follows. In Section \ref{section.the.model}, the Mathematical model is introduced. Following that, Sections \ref{section.definition.scale.function} and \ref{section.additional.assumption} briefly review some results in fluctuation theory for L\'evy processes, and the well-known sufficient optimality result in the literature, respectively. A verification lemma is then presented in Section \ref{section.verification.snlp}. In Sections \ref{section.value.choice.snlp}-\ref{section.optimal.snlp}, a candidate strategy is constructed and proved to be optimal. Convergence results for $\kappa\downarrow 0$ is shown in Section \ref{S.conv.kappa}. Section \ref{section.numerical} illustrates and Section \ref{section.conclusion} concludes.

\section{The model}\label{section.the.model}
In this paper we use the standard set-up for stochastic processes \citep*[e.g.][Chapter O]{Ber98}. A spectrally negative L\'evy process on the real line $Y=\{Y(t);t\geq 0\}$ is defined through its characteristic exponent, i.e.
\begin{equation}\label{def.snlp.1}
\mathbb{E}[e^{\theta Y(t)}]=e^{t\psi_Y(\theta)}
\end{equation}
and
\begin{equation}\label{def.snlp.2}
\psi_Y(\theta)= c\theta+\frac{\sigma^2}{2}\theta^2+\int_{(-\infty,0)}(e^{\theta s}-1-\theta s 1_{\{s>-1\}})\corr{\Upsilon}(ds),
\end{equation}
with
\begin{equation}\label{def.snlp.3}
\int_{(-\infty,0)}(1\wedge z^2)\corr{\Upsilon}(dz)<\infty,
\end{equation}
where $(c,\sigma,\corr{\Upsilon})$ are the L\'evy triplet of $Y$. In order to avoid trivial cases, we also require that $Y$ does not have a monotonic path. In this paper, we will use $\mathbb{P}_x$ and $\mathbb{E}_x$ to denote the probability measure and expectation for quantities for $X:=x+Y$ instead of $Y$, for $x\in\mathbb{R}$. For example $\mathbb{P}_x(X\in B):=\mathbb{P}(x+Y\in B)$ for a ``measurable'' set $B$. Note in particular we have ($\mathbb{P}$-a.s.) $X(0)=x$.

Periodic dividend decision (payment) times, or in short decision times, are the times when the Poisson process (independent of $X$) with rate $\gamma$, $N_\gamma(t)$, jumps from $i-1$ to $i$, i.e. the set $\mathbb{T}=\{T_i,i\in\mathbb{N}\}$ with
\begin{equation}
T_i=\inf\{t\geq 0: N_\gamma(t)=i | N_\gamma(0)=0\}.
\end{equation}
Let $\mathbb{F}:=\{\mathscr{F}(t);t\geq 0\}$ be the filtration generated by the process $(X,N_\gamma)$. Then, a periodic dividend strategy $\pi:=\{D^\pi(t);t\geq 0\}$ is a non-decreasing, right-continuous and $\mathbb{F}$-adapted process where the cumulative amount of dividends $D^\pi=\{D^\pi(t);t\geq 0\}$ admits the form
\[ D^\pi(t) = \int_{[0,t]}\nu^\pi(s)dN_\gamma(s),~t\geq 0,\quad D^\pi(0)=0.\]

Hence, the dividend amount paid at $T_i$ is $\xi^\pi_i:=\nu^\pi(T_i)$ (the increment of $D^\pi$ at $T_i$) and the strategy $\pi$ can also be specified in terms of $\{\xi^\pi_i;i\in\mathbb{N}\}$. The modified surplus $X^\pi=\{X^\pi(t);t\geq 0\}$ is defined as
\begin{align}
X^\pi(t)=X(t)-D^\pi(t)
\end{align}
and the ruin time $\tau^\pi$ is defined as
\begin{align}
\tau^\pi=\inf\{t\geq 0:X^\pi(t)< 0\},
\end{align} 
with the convention
\begin{equation}
\inf\emptyset =\infty.
\end{equation}

We now introduce some constraints for a periodic dividend strategy to be admissible. Since we are not allowed to inject capital to the company and a dividend payment cannot exceed the current surplus, a periodic strategy $\pi$ is admissible if it satisfies the following restriction:
\begin{equation}
X^\pi(T_i)\geq 0,~\forall ~T_i<\tau^\pi,~i\in\mathbb{N}.
\end{equation}
Intuitively, given that a fixed transaction cost $\kappa>0$ is incurred on each dividend payment, the amount of dividend should be large enough to pay the transaction cost, i.e.
\begin{equation}\label{Pikappa}
\xi_i^\pi \geq\kappa~\mbox{if}~\xi^\pi_i\neq 0.
\end{equation}
This holds naturally (see property 4 in Remark \ref{remark.different.kappa.on.v} below).

We can see from the above definitions that not paying any dividend is also allowed. In this case,  no transaction cost is incurred. We denote $\Pi$ the set of all admissible strategies and $\Pi_{\kappa}$ the set of all admissible strategy such that (\ref{Pikappa}) holds. Note when $X$ is of unbounded variation (i.e. with diffusion), it is possible that a dividend payment can cause ruin, which refers to liquidation of the company, i.e. the company chooses to close its business by distributing all the available surplus.

Lastly, we introduce the time preference parameter $\delta>0$. The value function of a strategy $\pi,~\pi\in\Pi$ with initial surplus $x$ is denoted as $V_{\kappa}(x;\pi)$ with the following definition:
\begin{align}
&V_{\kappa}(x;\pi)=\mathbb{E}_x\Big[\sum_{i=1}^{\infty}e^{-\delta T_i}(\xi_i^\pi-\kappa)1_{\{\xi_i^\pi>0\}}1_{\{T_i\leq\tau^\pi\}}\Big].\label{def.value.fcn.kappa}
\end{align}
Our goal is to find an optimal strategy $\pi^*_{\kappa}$ (if exists) such that
\begin{align}\label{def.optimal.value.function}
V_{\kappa}(x;\pi_{\kappa}^*)=v_{\kappa}(x):=\sup_{\pi\in\Pi} V_{\kappa}(x;\pi)(\geq 0).
\end{align}

\begin{remark}\label{remark.different.kappa.on.v}
	From the definitions of $V_{\kappa}$ and $\Pi_\kappa$, we have for any $0\leq\kappa_1\leq \kappa_2$
	\begin{enumerate}
		\item $\pi\in\Pi_{\kappa_2}\implies\pi\in\Pi_{\kappa_1}$ and
		\item $\pi\in\Pi_{\kappa_2}\implies V_{\kappa_1}(x;\pi)\geq V_{\kappa_2}(x;\pi)$ for all $x\geq 0$, and
		\item $v_{\kappa_1}(x)\geq v_{\kappa_2}(x)$ for all $x\geq 0$, and
		\item $V_{\kappa}(x;\pi)\geq 0$ for all $x\geq 0$ and $\pi\in\Pi_\kappa$,
		\item $v_{\kappa}(x)=\sup_{\pi\in\Pi_\kappa} V_{\kappa}(x;\pi)$.
	\end{enumerate}
Further justification for item 5. is provided in \citet*{AvLaWo20d}.
\end{remark}

Thanks to the fifth property in Remark \ref{remark.different.kappa.on.v}, it is sufficient to only consider the strategies in $\Pi_\kappa$. Therefore, in the remaining of this paper, we restrict ourselves to strategies in $\Pi_{\kappa}$.

Note that ruin is immediate when $X(0)=x<0$, which implies for any strategy $\pi$. 
\begin{equation}\label{Veq0.xleq0}
V_\kappa(x;\pi)=0,~x<0.
\end{equation}

\begin{definition}[Periodic ($b_u,b_l$) strategy]\label{D_bubl}
	A periodic $(b_u,b_l) $ strategy with $0\leq b_l\leq b_u$ is the strategy that pays $x-b_l$ whenever the surplus $x$ is above or equal to $b_u$, at decision times. This reduces the surplus level to $b_l$.
	
	\begin{figure}[H]
		\centering
		\includegraphics[width=0.7\textwidth]{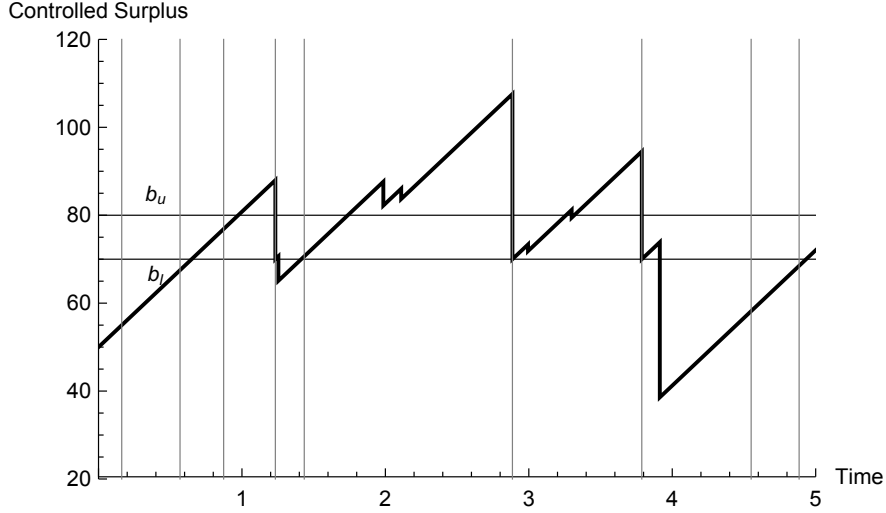}
		\caption{An illustration of a periodic $(b_u,b_l)$ strategy. The vertical lines represent the (Poissonian) dividend decision times.}
		\label{fig.bubl}
	\end{figure}
\end{definition}

By denoting the strategy as $\pi_{b_u,b_l}$, we have
\begin{align}\label{divpis}
\xi_i^{\pi_{b_u,b_l}}=[X^{\pi_{b_u,b_l}}(T_i-)-b_l]1_{\{X^{\pi_{b_u,b_l}}(T_i-)\geq b_u\}}.
\end{align}
Clearly, we have $\pi_{b_u,b_l}\in\Pi_\kappa\iff b_u-b_l\geq\kappa$.

\begin{definition}
	Similarly, a periodic barrier strategy at barrier level $b>0$, denoted as $\pi_{b}$, is defined as
	\begin{align}\label{divpbs}
	\xi_i^{\pi_b}=[X^{\pi_{b}}(T_i-)-b]1_{\{X^{\pi_b}(T_i-)\geq b\}}.
	\end{align}
\end{definition}

\section{Scale functions}\label{section.definition.scale.function}

This section very briefly review knowledge of (fluctuation theory of spectrally negative L\'evy processes and) scale functions for our purpose, i.e. to calculate the value function of a periodic $(b_u,b_l)$ strategy. Interested reader can refer to standard textbook such as \citet*{Ber98} and \citet*{Kyp06}. The tool we are going to use is fluctuation theory for L\'evy processes which is quite standard within the community. Specifically, we will borrow the recent results from the Section 6 of \citet*{PeYa16b}. To fully understand the results, we recommend the work of \citet*{Ber98}, \citet*{Kyp14}, \citet*{LoReZh14}, \citet*{Pan17}, \citet*{ChDo05}, \citet*{PaPeRi15}, \citet*{AvPeYa18} and \citet*{PeYa16b} (in the order), together with the references therein. 

The $q$-scale function, $W_q$, for $x\geq0$, $q\geq 0$ is defined through the inverse Laplace transform of $\frac{1}{\psi(\theta)-q}$, i.e.
\begin{equation*}
\int_{0}^{\infty}e^{-\theta x}W_q(x)dx=\frac{1}{\psi(\theta)-q},~~\theta>\phi_q,
\end{equation*}
where 
\begin{equation*}
\phi_q = \sup \{s\geq 0:\psi(s)=q\}.
\end{equation*}
Next, for $x\geq0$, the ``tilted'' $q$-scale function is define as
\begin{equation}
Z_q(x,\theta)=e^{\theta x}\Big(1-(\psi(\theta)-q)\int_{0}^{x}e^{-\theta y}W_q(y)dy\Big).
\end{equation}
In particular, when $\theta=0$, we write
\begin{align}
Z_q(x)=~&Z_q(x,0)=1+q\int_{0}^{x}W_q(y)dy.
\end{align} 
When $\theta=\phi_{q+r}$, we define
\begin{equation}
Z_{r,q}(x):=Z_q(x,\phi_{r+q})=r\int_0^\infty e^{-\phi_{r+q}u}W_q(x+u)du,
\end{equation}
where the last equality comes from $\int_0^\infty e^{-\phi_{r+q}u}W_q(u)du=1/r$.
Lastly, we define for $b\geq 0$
\begin{equation}\label{eqt.W3}
W_{r,q,b}(x):= W_q(x)+r\int_{b}^{x}W_{q+r}(x-y)W_q(y)dy,
\end{equation}
where the integral vanishes when $x\leq b$.

Notice that when $x<0$, all functions defined above are extended with $$W_q(x)=0,$$ which implies 
\begin{align}
Z_q(x)&=1,\\
{\overline{Z}}_q(x)&=x.
\end{align}
We also define the integral of functions by adding an overhead line to it, e.g. 
\begin{align*}
{\overline{W}}_q(x)&=\int_{0}^{x}W_q(y)dy,\\
{\overline{\overline{W}}}_q(x)&=\int_{0}^{x}{\overline{W}}_q(y)dy,\\
{\overline{Z}}_q(x)&=\int_{0}^{x}Z_q(y)dy,\nonumber
\end{align*}
for $x\geq 0$.

We end this section with a remark that except in a few cases explicit calculation of $W_q$ is difficult, if not impossible. We refer to Remark 1.1 in \citet*{LoReZh14} for a review in the numerical aspect of scale functions.

\section{Additional assumption for optimality}\label{section.additional.assumption}

It is well known that for a spectrally negative L\'evy process, barrier type of strategy is in general not optimal, e.g. see \cite*{AzMu05}. However, with the additional assumption that the L\'evy measure has completely monotonic density, the shape of the scale function is ``nice'' and barrier type of strategy is optimal, see e.g. \citet*{Loe08}, \citet*{Loe08a}, \citet*{NoPeYaYa17}.

In this section, we follow the lines of \citet*{Loe08} and assume that the L\'evy measure of $X$ has completely monotonic density, i.e. the L\'evy measure of the dual process $-{X}$, $\overline{\corr{\Upsilon}}$, admits a density $\eta$, whose $n^{th}$ derivative, $\eta^{(n)}$, exists for all $n\in\mathbb{N}$ with
\begin{equation}
{(-1)}^n\eta^{(n)}(x)\geq 0,~x>0.
\end{equation}

Note that this is a known sufficient condition for a (continuously decided) barrier strategy to be optimal. In general, the optimal strategy for a Cram\'er-Lundberg model is a band strategy as discovered by \citet*{AzMu05}, where many different ``bands'' are possible. Although being artificial, we note that if the optimal strategy is a barrier strategy, the optimal strategy is then characterised by the barrier level(s), which offers much simplicity to obtain qualitative insights.

Lastly, note that while a ``log-convex'' condition on the L\'evy measure \citep*[see][]{LoRe10} \corrr{is weaker, and} would imply a completely monotonic L\'evy density (which is the assumption we will make throughout the paper), we chose to directly assume the latter because it offers an explicit formula for the calculation of the first derivative of the value function in the upper branch (see Lemma \ref{lemma.10}).

\section{Preliminary results}\label{section.prelim.snlp}
From \citet*{NoPeYaYa17}, we know that the value function of a periodic barrier strategy at barrier level $b\geq 0$, ${\pi}_{b}$ is given by
\begin{align}
&{{V}}_0(x;{\pi}_b)=\frac{\gamma}{\phiqr\Zqr^\prime(b)}\Big(\WB(x)-\gamma\Wq(b)\Wb(x-b)\Big)-\gamma\Wbb(x-b),~x\in\mathbb{R},\\
&{{V}}_0(x;{\pi}_b)=\frac{\gamma}{\phiqr\Zqr^\prime(b)}\Wq(x),\quad x\leq b\label{value.b.lower}
\end{align}
when there are no fixed transaction costs, i.e. $\kappa=0$, \corr{where we recall the function $\WB$ is defined in \eqref{eqt.W3}}. In addition, when the L\'evy measure of $X$ has completely monotonic density, 
\begin{enumerate}
	\item We have $W_q\in\mathscr{C}^\infty$ and
	\begin{equation}
	W_q(x)=\frac{\partial}{\partial y}\phi(y)\Big|_{y=q}e^{\phi_q x}-\int_{0}^{\infty}e^{-x t}\mu_q(dt),
	\end{equation}
	for some finite measure $\mu_q$.
	
	\item We have $\Wq^{\prime\prime\prime}>0\implies$ there exists $\bar{b}\geq 0$ such that $\Wq^{\prime\prime}(b)<0, ~b\in(0,\bar{b})$,  $\Wq^{\prime\prime}(\bar{b})=0$ and $\Wq^{\prime\prime}(b)>0,~ b\in(\bar{b},\infty)$.
	
	\item For $x\geq b$, ${{V}}'_0(x;{\pi}_b)$ is given by
	\begin{equation}
	{{V}}_0^\prime(x;{\pi}_b)=K+\gamma\frac{\int_{0}^{\infty}e^{-tx}g(t,b)\mu_{q+r}(dt)}{\phi_{q+r}\Zqr^\prime(b)},
	\end{equation}
	where $K\in(0,1)$ and 
	\begin{equation}
	g(t,b)=t+\gamma\Wq(0)+\gamma\int_{0}^{b}e^{ut}\big(\Wq^\prime(u)-\frac{\phi_{q+r}}{\gamma}\Zqr^\prime(b)\big)du-\frac{\phi_{q+r}}{t}\Zqr^\prime(b).
	\end{equation}
	\item By defining the function $h$ as
	\begin{equation}
	h(x):=e^{-\phiqr x}\Zqr''(x)=\gamma\int_{x}^{\infty}e^{-\phiqr y}\Wq^{\prime\prime}(y)dy,~x>0,
	\end{equation}
	where the second equality is equation (4.11) in \citet*{NoPeYaYa17}, we have that
	\begin{equation}\label{result1.noba}
	\mbox{either $h(x)\geq 0$ for $x\geq 0$ or there exists a ${b}^*>0$ such that $h(x)\leq 0\iff x\in(0,{b}^*]$}.
	\end{equation}
	\item We have ${b}^*>0\iff h(0)\geq 0$. Moreover, it holds that
	\begin{equation}
	{{v}}_0(x)={{V}}_0(x;{\pi}_{{b}^*}),~x\geq 0
	\end{equation}
	and
	\begin{equation}\label{ineq.vDbstar}
	{{V}}_0^\prime(x;{\pi}_{{b}^*})<1, ~x\geq {b}^*.
	\end{equation}
	
\end{enumerate}

\section{Verification lemma}\label{section.verification.snlp}
A function $F$ is said to be smooth if $F\in\mathscr{C}^2$ (resp. $F\in\mathscr{C}^1$) if ${X}$ is of unbounded variation (resp. bounded variation). The extended generator for ${X}$, ${\mathscr{L}}$, applied on a function $F$ is given by
\begin{equation}\label{Def.extended.operation.2}
{\mathscr{L}}F(x):=cF'(x)+\frac{\sigma^2}{2}F''(x)+\int_{(-\infty,0)}\big[F(x+s)-F(x)-F'(x)s1_{\{|s|<1\}}\big]\corr{\Upsilon}(ds)
\end{equation}
if it is well defined, where the term $\frac{\sigma^2}{2}F''(x)$ is understood to be vanished if ${X}$ is of bounded variation (no Gaussian component). 
The following lemma characterises sufficient conditions for a candidate strategy to satisfy in order to be optimal.

\begin{lemma}\label{lemma.ver}
	Suppose $\pi\in{\Pi}_\kappa$ and its value function $H(x):={V}_{\kappa}(x;\pi)$ satisfies
	\begin{enumerate}
		\item $H$ is smooth, 
		\item $H\geq 0$ on $\mathbb{R}$,
		\item $(\mathscr{L}-\delta)H(x) +\gamma\corr{\max_{l\in[0,x]}}\Big(\corr{(l-\kappa)1_{\{l>0\}}}+ H(x-l)-H(x)\Big) \leq 0$, \corr{$x\geq 0$}.
	\end{enumerate}
	Then the strategy $\pi$ is optimal, i.e. $H(x)={v}_\kappa(x)$ for all $x\in\mathbb{R}$.
\end{lemma}
The proof follows from Appendix B in \citet*{AvLaWo20d}.

\section{Value function and the choice of $(b_u,b_l)$}\label{section.value.choice.snlp}

\subsection{On the value function and its smoothness}

The value function of a periodic $(b_u,b_l)$ strategy is given by the following lemma.
\begin{lemma}\label{lemma.value.snlp}
 \corr{The value function of a periodic $(b_u,b_l)$ strategy is given by
\begin{align}
{V}_\kappa(x;{\pi}_{b_u,b_l})=~& \frac{\gamma(\frac{1}{\phiqr}+g-\kappa)}{\phiqr\Zqr(b_u)-\gamma\Wq(b_l)}\Big(\WB(x)-\gamma\Wq(b_l)\Wb(x-b_u)\Big)\nonumber\\&-\gamma\Big(\Wbb(x-b_u)+(g-\kappa) \Wb(x-b_u)\Big),
\end{align}
where 
\begin{equation}
g:=b_u-b_l
\end{equation}
and we recall that $W_{\gamma,\delta,b}$ is defined in \eqref{eqt.W3}.}

If furthermore the smoothness condition (defined as)
\begin{equation}\label{smooth.condition}
{V}_\kappa(b_u;{\pi}_{b_u,b_l})={V}_\kappa(b_l;{\pi}_{b_u,b_l})+b_u-b_l-\kappa
\end{equation}
holds, the value function of a periodic $(b_u,b_l)$ strategy reduces to the value function of a periodic barrier strategy at barrier level $b_u$ without transaction costs, i.e. 
\begin{equation}\label{eqt.reduce.to.barrier}
{V}_\kappa(x;{\pi}_{b_u,b_l})={V}_0(x;{\pi}_{b_u})=\frac{\gamma}{\phiqr\Zqr^\prime(b_u)}\Big(\WB(x)-\gamma\Wq(b_u)\Wb(x-b_u)\Big)-\gamma\Wbb(x-b_u).
\end{equation}

\end{lemma}
\begin{proof}
Results from \citet*{PeYa16b} can be adapted to our context. See Appendix \ref{appendix.value.snlp} for details. 
Note that our proof does not use the additional assumption on the L\'evy measure, thus is true for general spectrally negative L\'evy processes.
\end{proof}

When $x\leq b_u$, the value function is given by
\begin{equation}\label{eqt.value.snlp.lower}
{V}(x;{\pi}_{b_u,b_l}) = \frac{\gamma(\frac{1}{\phiqr}+g-\kappa)}{\phiqr\Zqr(b_u)-\gamma\Wq(b_l)}\Wq(x),~x\leq b_u.
\end{equation}
Hence, the smoothness condition ${V}_\kappa(b_u;{\pi}_{b_u,b_l})={V}_\kappa(b_l;{\pi}_{b_u,b_l})+g-\kappa$ is equivalent to
\begin{align*}
&\frac{\gamma(\frac{1}{\phiqr}+g-\kappa)}{\phiqr\Zqr(b_u)-\gamma\Wq(b_l)}(\Wq(b_u)-\Wq(b_l))=g-\kappa\\
\iff~&\gamma(\frac{1}{\phiqr}+g-\kappa)(\Wq(b_u)-\Wq(b_l)) -(g-\kappa)(\phiqr\Zqr(b_u)-\gamma\Wq(b_l))=0\\
\iff ~& \frac{\gamma}{\phiqr}(\Wq(b_u)-\Wq(b_l))+(g-\kappa)(\gamma\Wq(b_u)-\phiqr \Zqr(b_u))=0\\
\iff~& (g-\kappa)\Zqr^\prime(b_u)-\frac{\gamma}{\phiqr}(\Wq(b_u)-\Wq(b_l))=0,
\end{align*}
or
\begin{equation}
{\Gamma}_{b_l}(g):=(g-\kappa)\Zqr^\prime(b_l+g)-\frac{\gamma}{\phiqr}(\Wq(b_l+g)-\Wq(b_l))=0,\label{eqt.Gamma}
\end{equation}
with $g=b_u-b_l$.

We want now to show that for all $b_l\geq 0$, there is a unique $g\geq \kappa$ such that ${\Gamma}_{b_l}(g)=0$, which is precisely the following lemma.
\begin{lemma}\label{lemma.bu.exist}
	For every $b_l\geq 0$, there exist a unique $b_u>b_l+\kappa$ such that the smoothness condition holds, i.e. ${\Gamma}_{b_l}(b_u-b_l)=0$.
\end{lemma}

\begin{proof}
	First, we have
	\begin{align*}
	{\Gamma}_{b_l}(\kappa)=~&(\kappa-\kappa)\Zqr^\prime(b_l+\kappa)-\frac{\gamma}{\phiqr}(\Wq(b_l+\kappa)-\Wq(b_l))\\
	=~&-\frac{\gamma}{\phiqr}(\Wq(b_l+\kappa)-\Wq(b_l))\\
	<~&0
	\end{align*}
	as $\Wq$ is an increasing function.
	
	We proceed to show
	 $\lim_{g\rightarrow\infty}{\Gamma}_{b_l}(g)=+\infty$.
	By differentiating (\ref{eqt.Gamma}) with respect to $g$ and denote $b_u=b_l+g$, we get
	\corr{\begin{align}\label{equation.dGamma}
	\frac{\partial}{\partial g}{\Gamma}_{b_l}(g)=~&
	(g-\kappa)\Zqr^{\prime\prime}(b_l+g)+\Zqr^\prime(b_l+g)-\frac{\gamma}{\phiqr}\Wq^\prime(b_l+g)\nonumber\\
	=~&(g-\kappa)\Zqr^{\prime\prime}(b_l+g)+\frac{1}{\phiqr}\Big(\phiqr\Zqr^\prime(b_l+g)-\gamma\Wq^\prime(b_l+g)\Big)\nonumber\\
	=~&(\frac{1}{\phiqr}+g-\kappa) e^{\phiqr b_u} h(b_u).
	\end{align}}
	Hence, by letting $g\rightarrow\infty$ and using $\Wq'''>0$ (the second item in Section \ref{section.prelim.snlp}), we have
	\corr{\begin{align}
	\lim_{g\rightarrow\infty}\frac{\partial}{\partial g}{\Gamma}_{b_l}(g)=~&\lim_{g\rightarrow\infty}(\frac{1}{\phiqr}+g-\kappa)\gamma e^{\phiqr b_u}\int_{b_u}^{\infty}e^{-\phiqr y}\Wq^{\prime\prime}(y)dy\nonumber\\
	>~&\lim_{g\rightarrow\infty}(\frac{1}{\phiqr}+g-\kappa)\gamma e^{\phiqr b_u}\int_{b_u}^{\infty}e^{-\phiqr y}dy\Wq^{\prime\prime}(y_0),~\mbox{large enough $g$ s.t. }\Wq^{\prime\prime}(y_0)>0\nonumber\\
	=~&+\infty
	\end{align}}
	and hence $\lim_{g\rightarrow\infty}{\Gamma}_{b_l}(g)=+\infty$ as desired. By the continuity of ${\Gamma}_{b_l}$ and equation \eqref{eqt.Gamma}, we know that there exists a root for ${\Gamma}_{b_l}(g)=0$. We now show that such root is unique.
	
From equation (\ref{equation.dGamma}), we know that
	\corr{\begin{equation*}
	\frac{\partial}{\partial g}{\Gamma}_{b_l}(g)=(\frac{1}{\phiqr}+g-\kappa) \Zqr^{\prime\prime}(b_u)=(\frac{1}{\phiqr}+g-\kappa)e^{\phiqr b_u}h(b_u).
	\end{equation*}}
	Hence, if ${\Gamma}_{b_l}(g)=0$ and $\frac{\partial}{\partial g}{\Gamma}_{b_l}(g)\leq 0$, we have 
	\begin{equation}
	\frac{\partial}{\partial g}{\Gamma}_{b_l}(g)\leq 0 \implies
	\begin{cases}
	&\Zqr^{\prime\prime}(b_u)\leq 0 \iff \Zqr^\prime(b_u)\leq \frac{\gamma}{\phiqr}\Wq^\prime(b_u)\\
	&b^*> 0\mbox{ exists and }\kappa\leq b_u\leq b^*<\bar{b}\implies \Wq^\prime(b_u)\leq \Wq^\prime(y),~y\leq b_u
	\end{cases},
	\end{equation}
	where \corr{we recall $\bar{b}$ is defined in the second item in Section \ref{section.prelim.snlp} and} the second implication follows from (\ref{result1.noba}). Consequently, we have
	\begin{align}
	{\Gamma}_{b_l}(g)
	=~&(g-\kappa)\Zqr^\prime(b_u)-\frac{\gamma}{\phiqr}(\Wq(b_u)-\Wq(b_l))\nonumber\\
	\leq~&(g-\kappa)\frac{\gamma}{\phiqr}\Wq^\prime(b_u)-\frac{\gamma}{\phiqr}(\Wq(b_u)-\Wq(b_l))\nonumber\\
	=~&\frac{\gamma}{\phiqr}\Big((g-\kappa)\Wq^\prime(b_u)-g\Wq^\prime(\alpha)\Big),~\corr{\exists\alpha\in[b_l,b_u]\text{ by Mean Value Theorem}}\nonumber\\
	=~&\frac{\gamma}{\phiqr}\Big((g(\Wq^\prime(b_u)-\Wq^\prime(\alpha))-\kappa\Wq^\prime(b_u)\Big)\nonumber\\
	<~&0,
	\end{align}
	which contradicts with the assumption that ${\Gamma}_{b_l}(g)=0$. Hence, we have that ${\Gamma}_{b_l}(g)=0$ implies $\frac{\partial}{\partial g}{\Gamma}_{b_l}(g)>0$, which show that the root for 
	${\Gamma}_{b_l}(g)=0$ is unique, by the continuity of ${\Gamma}_{b_l}$. Moreover, the root is also continuous in $b_l$ because $\Gamma_{b_l}$ defined in (\ref{eqt.Gamma}), as a 2 parameter function, is continuous in $(b_u,b_l)$.
\end{proof}
\begin{remark}\label{remark.bu.geq.bstar}
	In view of the proof of Lemma \ref{lemma.bu.exist}, we have 
	\begin{equation}
{\Gamma}_{b_l}(b_u-b_l)=0\implies \frac{\partial}{\partial g}{\Gamma}_{b_l}(g)\bigr|_{g=b_u-b_l}>0\implies h(b_u)>0\implies b_u>b^*,
	\end{equation}
	which also implies that when the smoothness condition is met, we have
	\begin{equation}\label{ZDD.geq.0}
	\Zqr''(b_u)>0.
	\end{equation}
\end{remark}

For a $(b_u,b_l)$ strategy such that the smoothness condition (\ref{eqt.Gamma}) holds, we also call the strategy ``smooth $(b_u,b_l)$ strategy'' and denote it as $\pis$. In particular, Lemma \ref{lemma.bu.exist} assures its existence. If $\Voo\leq 1$, the smooth $(b_u,b_l)$ strategy is also called ``optimal $(b_u,b_l)$ strategy''. Otherwise, if $\Voo> 1$ and  $V_\kappa^\prime(b_l;\pis)=1$, we call the smooth $(b_u,b_l)$ strategy ``optimal $(b_u,b_l)$ strategy''. The notation for an optimal $(b_u,b_l)$ strategy is $\pio$. 

In the remaining of this paper, unless otherwise specified, when considering the properties of a smooth (resp. optimal) $(b_u,b_l)$ strategy, we assume that the barriers and the fixed transaction costs $\kappa$ are given. In this spirit, we denote its value function $\Vs$ (resp. $\Vo$). If the dependence on the barriers or the transaction costs need to be stressed, we write the value function explicitly as $V_\kappa(\cdot;\pis)$ (resp. $V_\kappa(\cdot;\pio)$).

\subsection{Existence of the lower barrier $b_l$ and liquidation at first opportunity strategies} \label{S_bl}

As explained earlier, if $X$ is of unbounded variation (e.g. if a diffusion component exists) then $b_l=0$ corresponds to a liquidation at first opportunity. This is because the surplus is ruined as soon as it reaches 0. On the other hand, if $X$ is of bounded variation then ruin does not occur when the surplus is 0 because of the spectrally negative nature of the surplus dynamics. These cases occur when $\Voo\leq 1$, as we have $\pi_{b_u,0}^{\kappa,*}$ by definition (which is $\pi_{b_u,0}^{\kappa,s}$). This is illustrated in Section \ref{sec.103}.

The existence of an ``optimal $(b_u,b_l)$ strategy'' requires more care when $\Voo> 1$. The following lemma asserts the existence of an optimal $(b_u,b_l)$ strategy.

\begin{lemma}\label{lemma.bl.exist2}
	If $\Voo> 1$, then there exist a $b_l\in(0,{b}^*)$ such that ${V}_\kappa^\prime(b_l;\pis)=1$.
\end{lemma}
\begin{proof}
	First, notice that $b_l\mapsto {V}_\kappa^\prime(b_l;\pis)$ is continuous in $b_l$ due to Lemma \ref{lemma.bu.exist} and \eqref{eqt.value.snlp.lower}. Hence, by continuity it suffices to show that 
	\begin{equation}
	{V}_{b_l}^\prime(b_l)<1~\mbox{for all }b_l\geq {b}^*,
	\end{equation}
	which is equivalent to
	\begin{align*}
	0<~&\phiqr \Zqr^\prime(b_u)-\gamma\Wq^\prime(b_l)\\
	=~&\phiqr \Zqr^\prime(b_l)-\gamma\Wq^\prime(b_l)+\phiqr (\Zqr^\prime(b_u)-\Zqr^\prime(b_l))\\
	=~&\Zqr\pp(b_l)+\phiqr (\Zqr^\prime(b_u)-\Zqr^\prime(b_l))
	\end{align*}
	for all $b_l\geq {b}^*$. This is true because $b_l\geq {b}^*\implies h(x)\geq 0,~x\geq b_l\implies  \Zqr\pp(b_l)\geq 0$ and $\Zqr'(b_u)>\Zqr'(b_l)$, implying that the sum is positive.
\end{proof}

We have now proved the existence of an optimal $(b_u,b_l)$ strategy. Combining Lemma \ref{lemma.bl.exist2} with Remark \ref{remark.bu.geq.bstar}, for a $\pio$, we have
\begin{equation}\label{ineq.bs}
b_l<{b}^*<b_u.
\end{equation}

\section{The derivative of the value function, $\Vo$}
The existence of a $\pio$ strategy was shown in the previous sections. In this section, we investigate the properties of the value function of a given $\pio$ to prepare for the proof of its optimality. Our goal is to show (\ref{der.vbl}). 

We start by showing that 
\begin{equation}
\Vo^\prime(x)>1,~ x<b_l\mbox{ if }b_l>0,
\end{equation}
which is a direct consequence of (\ref{ineq.bs}). To be more specific, $\Wq$ is decreasing on $[0,\bar{b}]$ and $b^*<\bar{b}$, implying that $\Wq$ is decreasing on $[0,b_l]$. For a periodic $(b_u,b_l)$ strategy, the lower branch of the value function given by (\ref{eqt.value.snlp.lower}) is proportional to $\Wq$, hence is decreasing to $\Vo'(b_l)=1$ on $[0,b_l]$. 

Next, We show that the following lemma holds.
\begin{lemma}\label{lemma.vblPbu.leq1}
The inequality $\Vo^\prime(b_u)<1$ holds.
\end{lemma}
\begin{proof}
	Using (\ref{eqt.reduce.to.barrier}), we have
	\begin{align*}
	\iff&\Vo^\prime(b_u)<1\\
	\iff&\phiqr \Zqr^\prime(b_u)>\gamma\Wq^\prime(b_u)\\
	\iff&\Zqr\pp(b_u)>0,
	\end{align*}
	where the last line is true from (\ref{ZDD.geq.0}).
\end{proof}

Due to the shape of $\Wq^\prime$, i.e. the second item in Section \ref{section.prelim.snlp}, we have the following corollary.
\begin{corollary}\label{Corr.vD.leq1}
For $x\in(b_l,b_u]$,
	\begin{equation}
	\Vo^\prime(x)<1.
	\end{equation}
\end{corollary}
\begin{proof}
	Here we have 2 cases. If $b_u\leq \bar{b}$, then the derivative of $\Wq$ is decreasing on $[0,b_u]$. Hence, by (\ref{eqt.value.snlp.lower}) $\Vo$ is also decreasing on $[0,b_u]$. In particular, we have $\Vo'(x)<\Vo'(b_l)\leq 1$ for $x\in(b_l,b_u]$.
	
	On the other hand if $b_u>\bar{b}$, then the derivative of $\Wq$ is decreasing on $(0,\bar{b})$ and increasing on $(\bar{b},b_u]$. Similar to the previous case, we have $\Vo'(x)<\max\{\Vo'(b_l),\Vo'(b_u)\}\leq 1$ for $x\in(b_l,b_u]$.
\end{proof}

Next, we want to show the derivative of the value function is less than one beyond $b_u$, which is the consequence of the following lemma because of (\ref{ineq.vDbstar}).

\begin{lemma}\label{lemma.10}
	Recall that ${V}_0(x;{\pi}_b)$ is the value function of a periodic barrier strategy at barrier level $b$ without transaction costs. For  $x\geq b_u(\geq b^*)$, we have
	\begin{equation}\label{lemma.vDleq1}
	\Vo^\prime(x)\leq {V}^\prime_0(x;{\pi}_{{b}^*}).
	\end{equation}
\end{lemma}
\begin{proof}
	Recall that for $x\geq b$, we have
	\begin{equation*}
	{V}_0^\prime(x;{\pi}_b)=K+\gamma\frac{\int_{0}^{\infty}e^{-tx}g(t,b)\mu_{q+r}(dt)}{\phi_{q+r}\Zqr^\prime(b)},
	\end{equation*}
	where $K\in(0,1)$ and 
	\begin{equation*}
	g(t,b)=t+\gamma\Wq(0)+\gamma\int_{0}^{b}e^{ut}\big(\Wq^\prime(u)-\frac{\phi_{q+r}}{\gamma}\Zqr^\prime(b)\big)du-\frac{\phi_{q+r}}{t}\Zqr^\prime(b).
	\end{equation*}
	Since $\pio$ is also a smooth $(b_u,b_l)$ strategy, we have from (\ref{eqt.reduce.to.barrier}) $\Vo(x)={V}_0(x;{\pi}_{b_u})$. Hence, in order to show (\ref{lemma.vDleq1}), it suffices to show that 
	\begin{equation*}
	\frac{g(t,b_u)}{\Zqr^\prime(b_u)}\leq \frac{g(t,{b}^*)}{\Zqr^\prime({b}^*)}
	\end{equation*}
	holds. Note that we have $\Zqr\pp(b_u)>0$ from (\ref{ZDD.geq.0}), which implies $\Zqr^\prime({b}^*)<\Zqr^\prime(b_u)$. Hence, it suffices to show 
	\begin{equation}
	\int_{0}^{b_u}e^{ut}\big(\Wq^\prime(u)-\frac{\phi_{q+r}}{\gamma}\Zqr^\prime(b_u)\big)du\leq \int_{0}^{{b}^*}e^{ut}\big(\Wq^\prime(u)-\frac{\phi_{q+r}}{\gamma}\Zqr^\prime({b}^*)\big)du.
	\end{equation}
	Recall that we have $V_\kappa^\prime(b_l;\pio)\leq 1$ by the definition of $\pio$. By direct computing using (\ref{value.b.lower}), we have 
	$-\phi_{q+r}\Zqr^\prime(b_u)\leq -\Wq^\prime(b_l)$. In addition, we have  $\Wq^\prime(b_l)>\Wq^\prime({b}^*)$ and $\Wq^\prime(u)\leq \Wq^\prime(b_l),~u\in[b_l,b_u]$, due to the shape of the scale function $\Wq$, see the second item in Section \ref{section.prelim.snlp}. If ${b}^*=0$, we have $b_l=0$ and hence
	\begin{align*}
	\int_{0}^{b_u}e^{ut}\big(\Wq^\prime(u)-\frac{\phi_{q+r}}{\gamma}\Zqr^\prime(b_u)\big)du=~&\int_{b_l}^{b_u}e^{ut}\big(\Wq^\prime(u)-\frac{\phi_{q+r}}{\gamma}\Zqr^\prime(b_u)\big)du\\\leq~& \int_{b_l}^{b_u}e^{ut}\big(\Wq^\prime(u)-\Wq^\prime(b_l)\big)du\\<~&0\\=~&\int_{0}^{{b}^*}e^{ut}\big(\Wq^\prime(u)-\frac{\phi_{q+r}}{r}\Zqr^\prime({b}^*)\big)du.
	\end{align*}
	On the other hand, if ${b}^*>0$, we have
	\corr{\begin{align*}
	\int_{0}^{b_u}e^{ut}\big(\Wq^\prime(u)-&\frac{\phi_{q+r}}{\gamma}\Zqr^\prime(b_u)\big)du
	~\leq~\int_{0}^{b_u}e^{ut}\big(\Wq^\prime(u)-\Wq^\prime(b_l)\big)du\\
	=~&\int_{0}^{b_l}e^{ut}\big(\Wq^\prime(u)-\Wq^\prime(b_l)\big)du+\int_{b_l}^{b_u}e^{ut}\big(\Wq^\prime(u)-\Wq^\prime(b_l)\big)du\\
	<~&\int_{0}^{b_l}e^{ut}\big(\Wq^\prime(u)-\Wq^\prime(b)\big)du+\int_{b_l}^{b_u}e^{ut}\big(\Wq^\prime(u)-\Wq^\prime(b_l)\big)du\\
	<~&\int_{0}^{b}e^{ut}\big(\Wq^\prime(u)-\Wq^\prime(b)\big)du+\int_{b_l}^{b_u}e^{ut}\big(\Wq^\prime(u)-\Wq^\prime(b_l)\big)du\\
	<~&\int_{0}^{{b}^*}e^{ut}\big(\Wq^\prime(u)-\Wq^\prime({b}^*)\big)du
	~=~\int_{0}^{{b}^*}e^{ut}\big(\Wq^\prime(u)-\frac{\phi_{q+r}}{\gamma}\Zqr^\prime({b}^*)\big)du.
	\end{align*}}
\end{proof}

To sum up, we have 
\begin{equation}\label{der.vbl}
\begin{cases}
{V}_{b_l^*}'(x)\leq 1,~x\geq 0,~&\mbox{if }b_l=0,\\
{V}_{b_l^*}'(x)\begin{cases}
&>1,~x<b_l\\
&=1,~x=b_l\\
&<1,~x>b_l
\end{cases},~&\mbox{if }b_l>0.
\end{cases}
\end{equation}
The first case ($b_l=0$) corresponds to a liquidation at first opportunity when $X$ is of unbounded variation (see Section \ref{S_bl}).

\section{Optimality}\label{section.optimal.snlp}

In this section, we verify that the strategy $\pio$ is optimal by arguing that all conditions in Lemma \ref{lemma.ver} are satisfied by its value function. 

First, recall from (\ref{eqt.reduce.to.barrier}) that when the smoothness condition is met, the value function at barrier level $b_u$ is the same as the value function of a periodic barrier strategy without transaction costs. It was shown in \citet*{NoPeYaYa17} that the value function of a periodic barrier strategy belongs to the class of $\mathscr{C}^2(0,\infty)$ (resp. $\mathscr{C}^1(0,\infty)$) if $X$ is of unbounded (resp. bounded) variation. Therefore, we conclude that the value function of the strategy $\pio$ is smooth (see Section \ref{section.verification.snlp} for the definition of smoothness). 
The second condition is satisfied directly from the definition of value function.
Conditions 3 can be shown to be met by proceeding in a similar fashion to Lemma 9.5 in \citet*{AvLaWo20d}, given the range of the derivative of the value function specified in (\ref{der.vbl}).

To conclude, we present the following theorem and corollary.
\begin{theorem}
The optimal strategy $\pio$ is optimal, i.e. $V_\kappa(s;\pio)=v_\kappa(x)$, $x\geq 0$.
\end{theorem}

\begin{corollary}\label{Coro.uniqueness}
	There is only 1 pair of $(b_u,b_l)$ which qualifies to be a $\pio$. We write the pair $({b}_u^*,{b}_l^*)$ and the strategy $\pi_{b_u^*,b_l^*}$ is optimal.
\end{corollary}
The proof follows from Proposition 11.3 in \citet*{AvLaWo20d}. Therefore, we can conclude that there are (unique) $(b_u^*,b_l^*)$ such that $\pi_{b_u^*,b_l^*}$ is optimal.

\section{\corr{Convergence result when $\kappa\downarrow 0$}}
\label{S.conv.kappa}

\corr{In this section, we show in Lemmas \ref{lemma.conv.kappa.bgeq0} and \ref{lemma.conv.kappa.beq0} that when $\kappa\downarrow 0$, $b_u^*-b_l^*\rightarrow 0$ so that combining with \eqref{ineq.bs}, it can be concluded that $b_u^*\rightarrow b^*$ and $b_l^*\rightarrow b^*$. In particular, convergence in barriers will imply convergence in both strategy $\pi$ (convergence of $D^\pi$ with probability one) and value function $V$. Note it can be shown that the convergence is monotonic (i.e. $b_u^*-b_l^*\downarrow 0$) with some additional steps. Since the latter implications can be obtained easily once the convergence of barriers is shown, we will only show the convergence of the barriers $b^*_u$ and $b_l^*$. For a numerical illustration, we refer to Figure \ref{fig.sen.kappa}.}

\corr{In order to show our main result, we need to first show the following lemma, which should hold intuitively.}

\corr{\begin{lemma}\label{lemma.vkappa.converge}
    We have $\lim_{\kappa\downarrow 0}v_\kappa(x)=v_0(x)$ for all $x\geq 0$.
\end{lemma}}
\corr{\begin{proof}
    By the definition in equation \eqref{def.value.fcn.kappa}, we see that
    $$V_0(x;\pi)=V_\kappa(x;\pi)+\kappa\mathbb{E}\sum_{i\in\mathbb{N}}e^{-\delta T_i}1_{\{T_i\leq \tau^\pi\}}$$
    for $\pi\in\Pi$. Therefore, by taking supremum, we have
    $$
        v_0(x)\leq \sup_{\pi\in\Pi}V_\kappa(x;\pi)+\kappa \mathbb{E}\sum_{i\in\mathbb{N}}e^{-\delta T_i}=v_\kappa(x)+\kappa \mathbb{E}\sum_{i\in\mathbb{N}}e^{-\delta T_i}\leq v_0(x)+\kappa \mathbb{E}\sum_{i\in\mathbb{N}}e^{-\delta T_i},
    $$
    where the term with $\kappa$ is positive and finite. Therefore, by taking $\kappa\downarrow 0$, it is readily deduced that
    $\lim_{\kappa\downarrow 0}v_\kappa(x)=v_0(x)$ for all $x\geq 0$.
\end{proof}}

\corr{We are now ready to show our main result when $b^*>0$ in the following lemma.}
\corr{\begin{lemma}\label{lemma.conv.kappa.bgeq0}
    Suppose $b^*>0$, we have that when $\kappa\downarrow 0$, $b_u^*-b_l^*\rightarrow 0$.
\end{lemma}}
\begin{proof}
    \corr{By the hypothesis that $b^*>0$ and Lemma \ref{lemma.vkappa.converge}, we have that 
    \begin{enumerate}
        \item $\Wq'$ is decreasing on $(0,\bar{b})$ which contains $(0,b^*)$;
        \item $d_{[0,b^*]}(v_\kappa,v_0)\rightarrow 0$ when $\kappa\downarrow 0$, where $d_{K}(\cdot,\cdot)$ is the distance induced by the uniform norm for continuous functions on a compact set $K$.
    \end{enumerate}}
    
    \corr{Now, suppose the statement in the lemma is false, then there is a $\varepsilon>0$ and a decreasing sequence $\{\kappa_n;n\in\mathbb{N}\}$ such that $\kappa_n\downarrow 0$ and $b_u^*(\kappa_n)-b_l^*(\kappa_n)>\varepsilon$ for all $n\in\mathbb{N}$. For convenience, we denote the barriers $(b_u^n,b_l^n)$.}
    
   \corr{ Let's first consider the case there are infinite many $n$ such that $b_l^n\geq \bar{\delta}>0$ so by taking a subsequence if necessary, we can work with $b_l^n\geq \bar{\delta}$. Since $$v_0(x)=V_0(x;\pi_{b^*})=\frac{\gamma}{\phiqr\Zqr'(b^*)}\Wq(x),\quad v_{\kappa_n}(x)=V_0(x;\pi_{b_u^n})=\frac{\gamma}{\phiqr\Zqr'(b_u^n)}\Wq(x)$$
    for $x\leq \bar{\delta}$, we see that $b_u^n\rightarrow b^*$ by the second item at the beginning of the proof. In particular, we can assume that $b_u^*=b^*+\delta_n$, where $\delta_n\downarrow 0$ (by passing a subsequence if necessary). This means we can assume that $b^*-b_l^n>\varepsilon_1$ (with $\varepsilon_1\in(0,\varepsilon)$) without loss of generality. By the second point above, for large enough $n$, we have
    $$v_{\kappa_n}(b^*)-v_0(b^*)=\varepsilon_2,\quad v_{\kappa_n}(b^*-\varepsilon_1)-v_0(b^*-\varepsilon_1)=\varepsilon_3$$
    for some $\varepsilon_2,\varepsilon_3>0$ arbitrarily small absolute values. This implies 
    $$v_{\kappa_n}(b^*)-v_{\kappa_n}(b^*-\varepsilon_1)=v_0(b^*)-v_0(b^*-\varepsilon_1)+\varepsilon_2-\varepsilon_3,$$
    or
    $$\int_{b^*-\varepsilon_1}^{b^*}v_{\kappa_n}'(y)dy=\int_{b^*-\varepsilon_1}^{b^*}v_{0}'(y)dy+\varepsilon_2-\varepsilon_3.$$
    Note from Corollary \ref{Corr.vD.leq1}, the integral on the left is strictly less than $\varepsilon_1$, while the integral on the right is strictly larger than $\varepsilon_1$ (see Section \ref{section.prelim.snlp}). Since $\varepsilon_2-\varepsilon_3$ can be made arbitrarily small (in absolution value) by increasing $n$, we arrive at a contradiction.}
    
    \corr{For the case where $b_l^n\downarrow 0$, the above argument applies if we choose $\varepsilon_1=(b^*/2)\vee (b^*-\varepsilon)$.} 

\end{proof}

\corr{The next lemma deals with the case $b^*=0$.}

\corr{\begin{lemma}\label{lemma.conv.kappa.beq0}
    Suppose $b^*=0$, we have that when $\kappa\downarrow 0$, $b_u^*-b^*_l\rightarrow 0$.
\end{lemma}}
\begin{proof}
    \corr{Note when $b^*=0$, we have that $b_l^*=0$ so essentially we need to show $b_u^*\rightarrow 0$. On the other hand, from $b^*=0$, we know that $\Wq'(0+)<\infty$. Similar to Lemma \ref{lemma.conv.kappa.bgeq0}, we work with a decreasing sequence of $\kappa$, say $\{\kappa_n,n\in\mathbb{N}\}$, with $\kappa_n\downarrow 0$.}
    
    \corr{Define the function
    \begin{equation}
        \widetilde{\Gamma}_\kappa(\varepsilon):=\varepsilon\Zqr'(\kappa+\varepsilon)-\frac{\gamma}{\phiqr}(\Wq(\kappa+\varepsilon)-\Wq(0)),
    \end{equation}
    which is the smoothness condition when $b_l=0$, see \eqref{eqt.Gamma}, when $b_u$ is decomposed to $b_l+\kappa+\varepsilon$. From this and that $\Zqr''>0$ on $(0,\infty)$ (because $b^*=0$), we have
    \begin{align*}
        \widetilde{\Gamma}_\kappa(\varepsilon)=~&
        \int_0^\varepsilon\Zqr'(\kappa+\varepsilon)dy - \frac{\gamma}{\phiqr}\int_0^\varepsilon \Wq'(\kappa+y)dy- \frac{\gamma}{\phiqr}\int_0^\kappa \Wq'(y)dy\\
        \geq~& \frac{1}{\phiqr}\int_0^\varepsilon \Zqr''(\kappa+y)dy- \frac{\gamma}{\phiqr}\int_0^\kappa \Wq'(y)dy.
    \end{align*}
    Since the second term goes to zero when $\kappa_n\downarrow 0$ and the term inside the first integral is strictly positive, we have $\varepsilon\rightarrow 0$ for $\widetilde{\Gamma}_\kappa(\varepsilon)=0$, when $\kappa_n\downarrow 0$. }

    \corr{This completes the proof.}
\end{proof}

\section{Numerical Illustrations}\label{section.numerical}

Our base model is a diffusion model with mixed exponential downward jumps, i.e.

$$X(t)=ct+\sigma W(t)-\sum_{k=1}^{N(t)}G_k,$$
where $W=\{W(t);t\geq 0\}$ is a standard Brownian motion, $N=\{N(t);t\geq 0\}$ is a Poisson process with rate $\lambda$, and where $G_k$ are sampled i.i.d. from a mixed exponential distribution, i.e.
$$\mathbb{P}(G_k\leq x)=p_1e^{-\beta_1 x}+p_2e^{-\beta_2 x},\quad x\geq 0,~\forall k$$
with $p_1+p_2=1$, $p_1,p_2,\beta_1,\beta_2>0$. 

In this case, the Laplace exponent (minus $q$) is given by
$$\psi(\theta)-q=c\theta+\frac{\sigma^2}{2}\theta^2+\lambda\Big(p_1\frac{\beta_1}{\beta_1+\theta}+p_2\frac{\beta_2}{\beta_2+\theta}\Big)-\lambda-q.$$
It is very easy to see that there are $4$ distinct real roots ($1$ positive and $3$ negative) for $\psi(\theta)-q=0$. Denote them $r_j,~j=1,2,3,4$ with order $r_i>r_j$ for $i<j$. Since the function $1/(\psi(\theta)-q)$ is a rational function, we can further express it (using partial fraction) as
$$\frac{1}{\psi(\theta)-q}=\sum_j \frac{A_j}{\theta-r_j}$$
with 
$$A_j=\frac{1}{\psi'(r_j)}.$$
Note $$\int_0^\infty e^{-\theta x}e^{r_j x}dx=\int_0^\infty e^{-(\theta-r_j)}dx=\frac{1}{\theta-r_j}$$
and from the uniqueness of Laplace transform, we can deduce that
$$W_q(x)=\sum_j A_j e^{r_j x}.$$
Therefore, all other scale functions can be computed explicitly easily.

To find the optimal barriers $(b_u^*,b_l^*)$, we make use of Lemmas \ref{lemma.bu.exist} and \ref{lemma.bl.exist2}. To be more specific, we perform the following:
\begin{enumerate}
	\item Find $b^*$ using (\ref{result1.noba}). Specifically, if $\Zqr^{\prime\prime}(0)\geq 0$, then set $b^*=0$, otherwise, solve $b^*$ such that $\Zqr^{\prime\prime}(b^*)= 0$. This can be done by (1) trying a large enough $b$ such that $\Zqr^{\prime\prime}(b)> 0$ following by (2) a bisection method on the range $[0,b]$.
	\item Write a function on $b_l\in[0,b^*]$ to output $b_u$ from Lemma \ref{lemma.bu.exist} with a similar method as the previous step (using range $[\max(\kappa,b^*),b]$ for large enough $b$), then calculate the derivative of the value function at $b_l$ and return this number. Say we call this function $G$.
	\item Find $b_l^*$ using Lemma \ref{lemma.bl.exist2}. Specifically, if $G(0)\leq 1$, then we set $b_l^*=0$, otherwise we can obtain $b_l^*$ by solving $G(b_l^*)=1$ via a bisection method on the range $[0,b^*]$. Use Lemma \ref{lemma.bu.exist} to calculate $b_u^*$ from $b_l^*$.
\end{enumerate}
\begin{remark}
	We remark that gradient descent type of methods typically do not work well here because a relatively large increment of the parameters (barriers) only results in a small change of the objective function (i.e. plateau). Therefore, analytic methods (those used in this paper) are needed. Perhaps more importantly, this shows that in practice one typically has more flexibility to deviate from the optimal strategy to incorporate other considerations.
\end{remark}

Our choice of parameters is $\lambda=10$, $p_1=0.9$, $\beta_1=1.9$, $\beta_2=0.19$, $c=11$, $\sigma=1$, $\gamma=1$, $\delta=0.2$ and $\kappa=0.2$. The Brownian motion term are used to model uncertainty (e.g. of the expenses), while the jump terms are used to model small and (ten times) larger claims, which occur at an average rate of 10\% of total claims.
Here, the value of $c$ is chosen such that the profit loading is $10\%$ and the expected profit per unit of time (net drift) of the process is $1$. In the following, we will use $\mu$ and $\varsigma^2$ to denote the expected value and the variance of the expected profit (increment over one unit of time of the surplus before dividends), respectively. That is,
\begin{align}
    \mu:=~&c-\lambda\Big(\frac{p_1}{\beta_1}+\frac{p_2}{\beta_2}\Big)=1,\\
    \varsigma^2:=~& \sigma^2+\frac{p_1}{\beta_1^2}+\frac{p_2}{\beta_2^2}.
\end{align}
In addition, we also use $M$ to denote the ratio between the expected value of the large claims and that of the small claims, i.e.
\begin{equation}
    M:=\frac{\beta_1}{\beta_2}=10.
\end{equation}

\subsection{The first derivative of the optimal value function beyond the upper barrier}

Generally speaking, we expect the optimal value function $v$ (see Equation \eqref{def.optimal.value.function}) to be concave because we expect the law of diminishing return holds. However, with the presence of fixed transaction costs, this is not necessarily the case. To see why it is possible to violate the concavity property, we shall consider the scenario when the surplus is high and the fixed transaction cost is also high. In such scenario, an incremental increase in surplus would actually decrease the ratio of the transaction costs to the first dividend payment (provided the company has not ruined yet), achieving a (relatively) higher return. Hence, the first order derivative $v'$ is increasing, as shown in Figure \ref{fig.vd.k2}. Remarkably, such a case also seems to hold even when the fixed transaction cost is low, see Figure \ref{fig.vd.k0.01}. We therefore conjecture that this would be the general case. We also plot the second order derivative for reference. The horizontal line in Figures \ref{fig.vd.k0.01} and \ref{fig.vd.k2} are the asymptotes $\gamma/(\gamma+\delta)$.

\begin{figure}[ht]
	\centering
	\begin{minipage}{0.3\textwidth}
		\centering
		\includegraphics[width=1\textwidth]{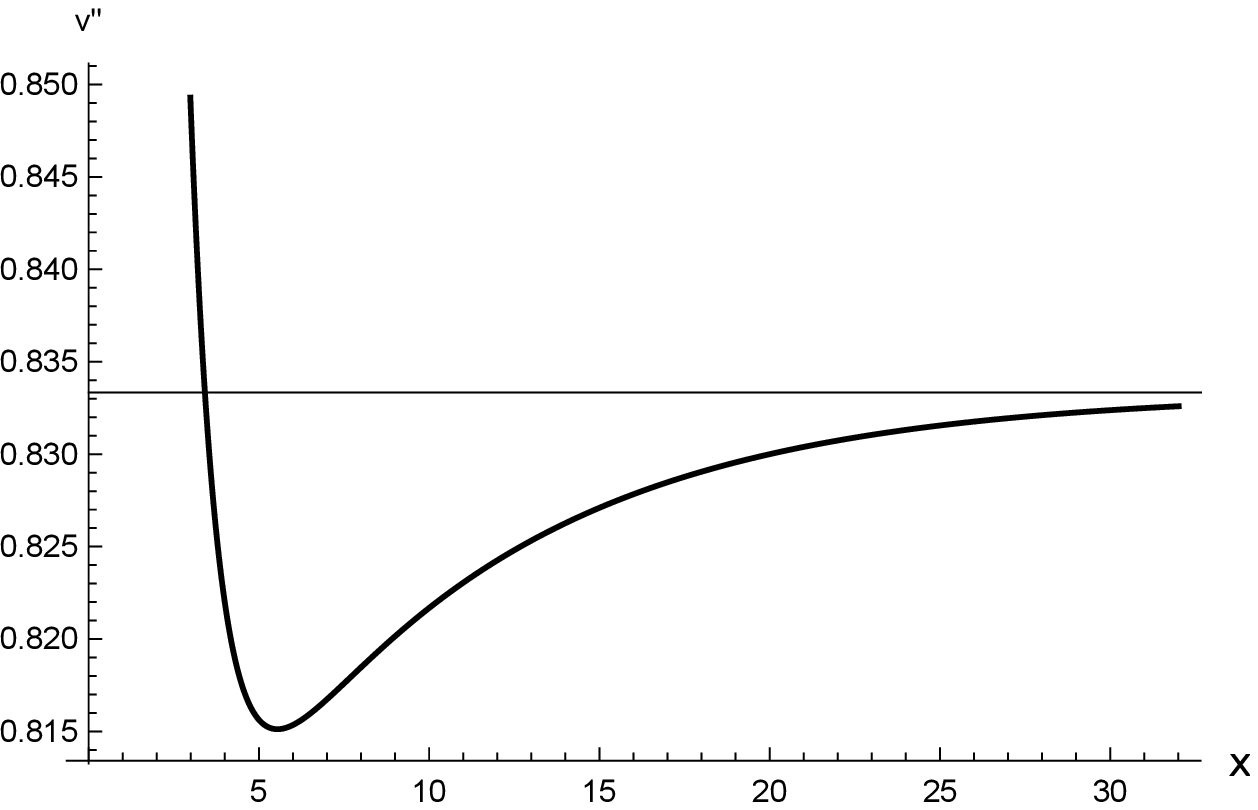}
		\subcaption[first caption.]{$v'(b_u^*+x)$ against $x$: $\kappa=0.01$}\label{fig.vd.k0.01}
	\end{minipage}%
	~~~\begin{minipage}{0.3\textwidth}
		\centering
		\includegraphics[width=1\textwidth]{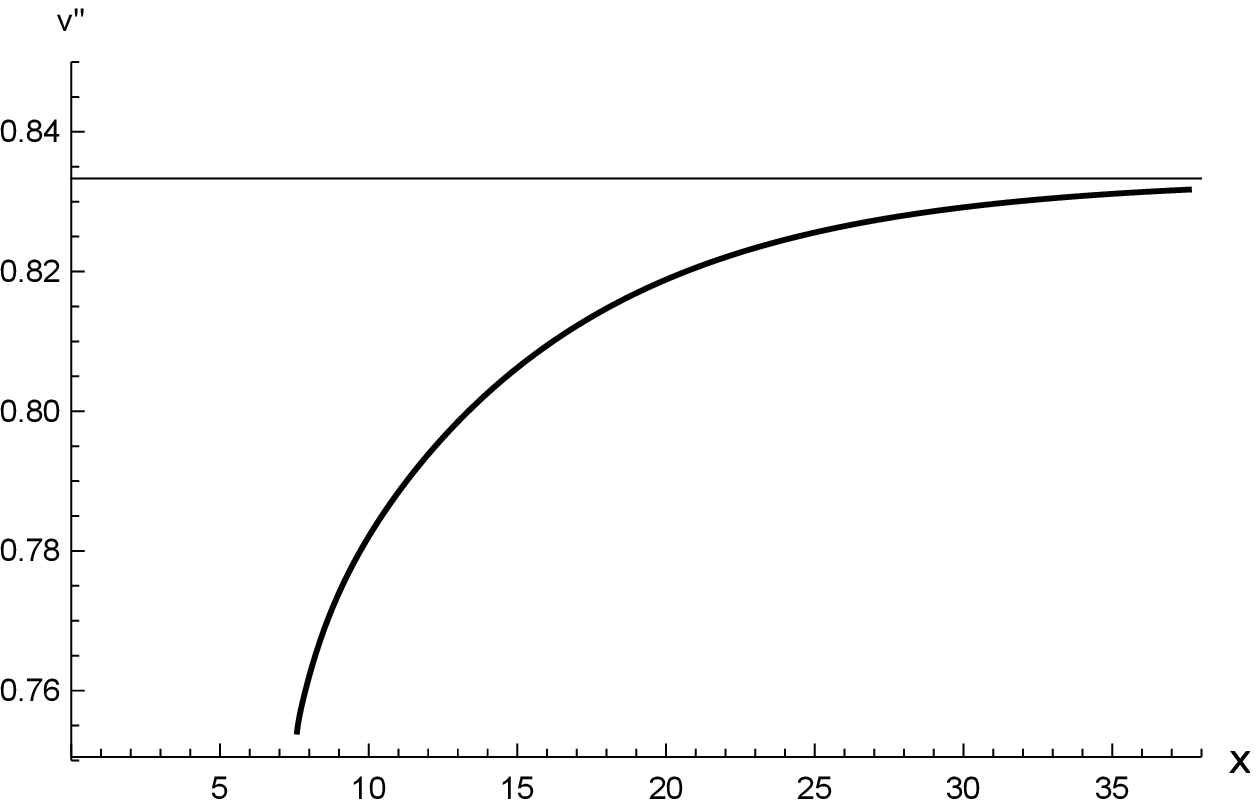}
		\subcaption[second caption.]{$v'(b_u^*+x)$ against $x$: $\kappa=2$}\label{fig.vd.k2}
	\end{minipage}%
	\\\begin{minipage}{0.3\textwidth}
		\centering
		\includegraphics[width=1\textwidth]{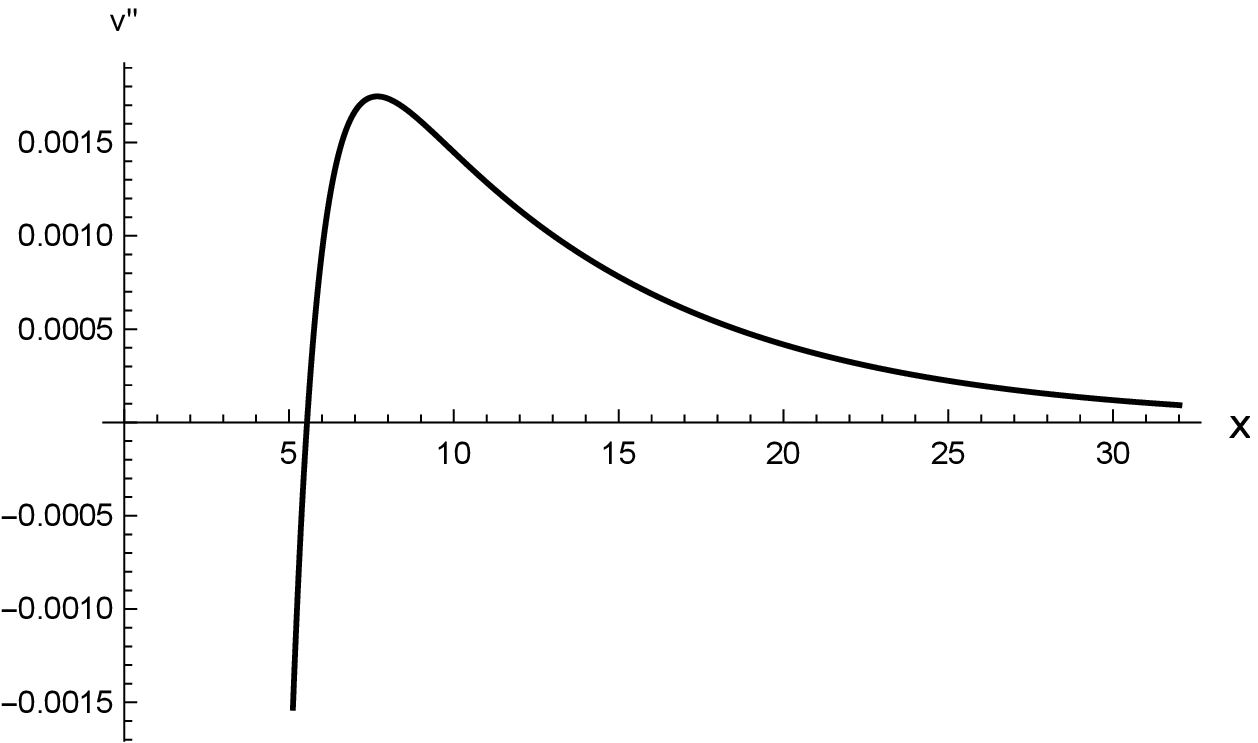}
		\subcaption[first caption.]{$v''(b_u^*+x)$ against $x$: $\kappa=0.01$}\label{fig.vdd.k0.01}
	\end{minipage}%
	~~~\begin{minipage}{0.3\textwidth}
		\centering
		\includegraphics[width=1\textwidth]{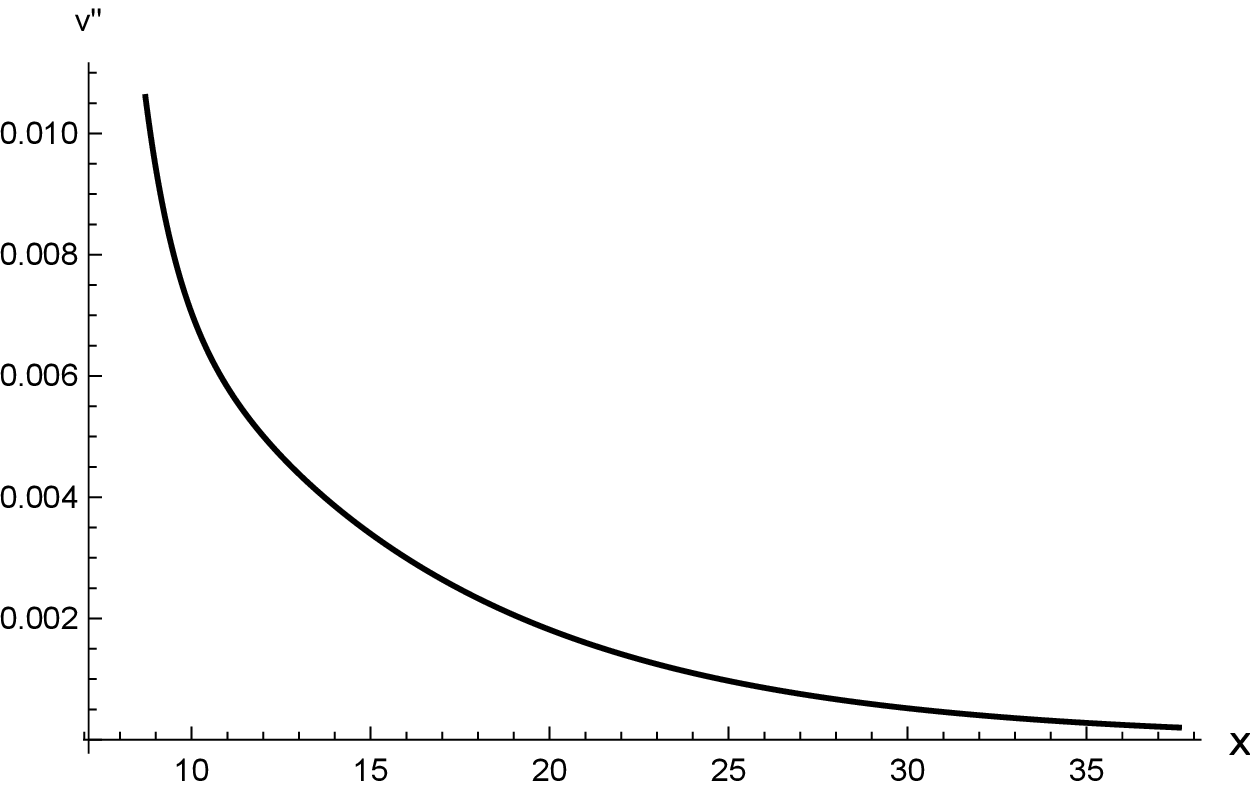}
		\subcaption[second caption.]{$v''(b_u^*+x)$ against $x$: $\kappa=2$}\label{fig.vdd.k2}
	\end{minipage}%
	\caption{The derivatives $v'(b_u^*+x)$ (top row) and $v''(b_u^*+x)$ (bottom row) against $x$} \label{fig.vdd}
\end{figure}

\subsection{The impact of risk on the optimal barriers}

\begin{figure}[htb]
	\centering
	\begin{minipage}{0.3\textwidth}
	\centering
	\includegraphics[width=1\textwidth]{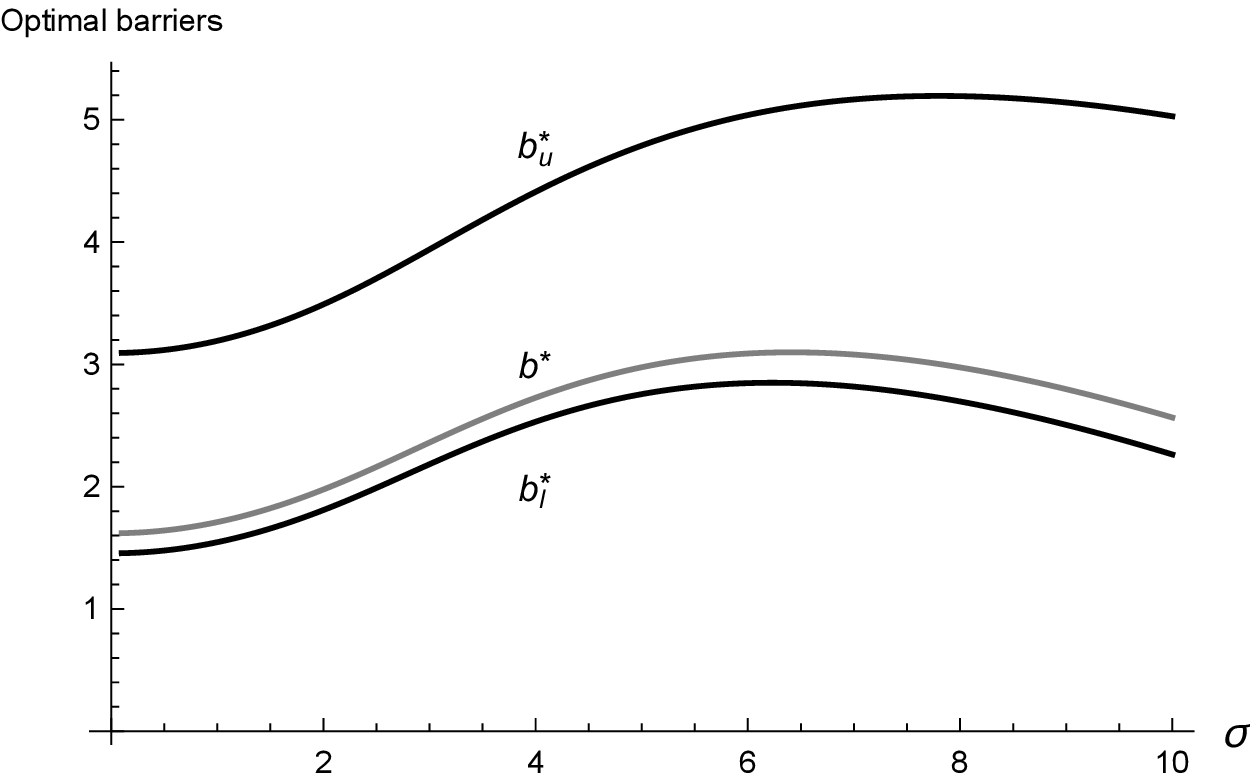}
	\subcaption[fourth caption.]{Impact of $\sigma$ on the barriers}\label{fig.sen.sigma}
\end{minipage}	
\\
	\begin{minipage}{0.3\textwidth}
		\centering
		\includegraphics[width=1\textwidth]{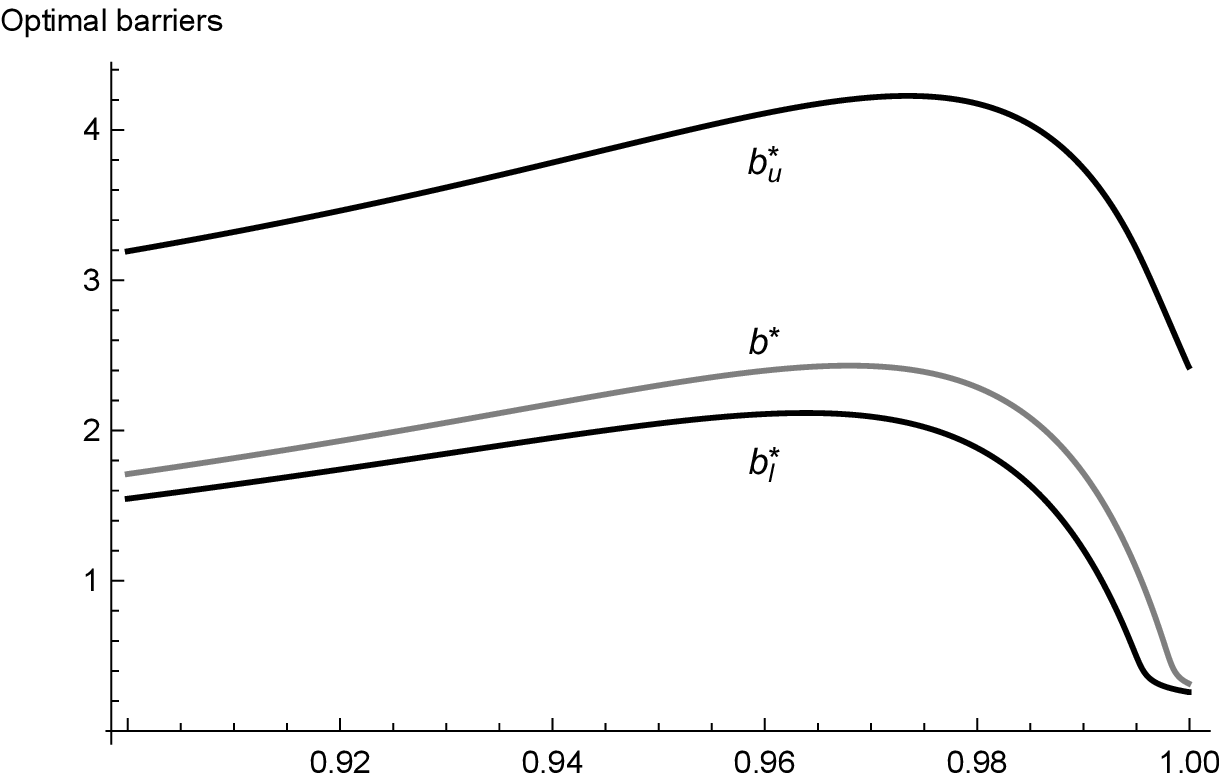}
		\subcaption[third caption.]{$x$-axis: Probability of small claims, with ratio between claims expectations $M$ fixed}\label{fig.sen.muP}
	\end{minipage}
	~~\begin{minipage}{0.3\textwidth}
		\centering
		\includegraphics[width=1\textwidth]{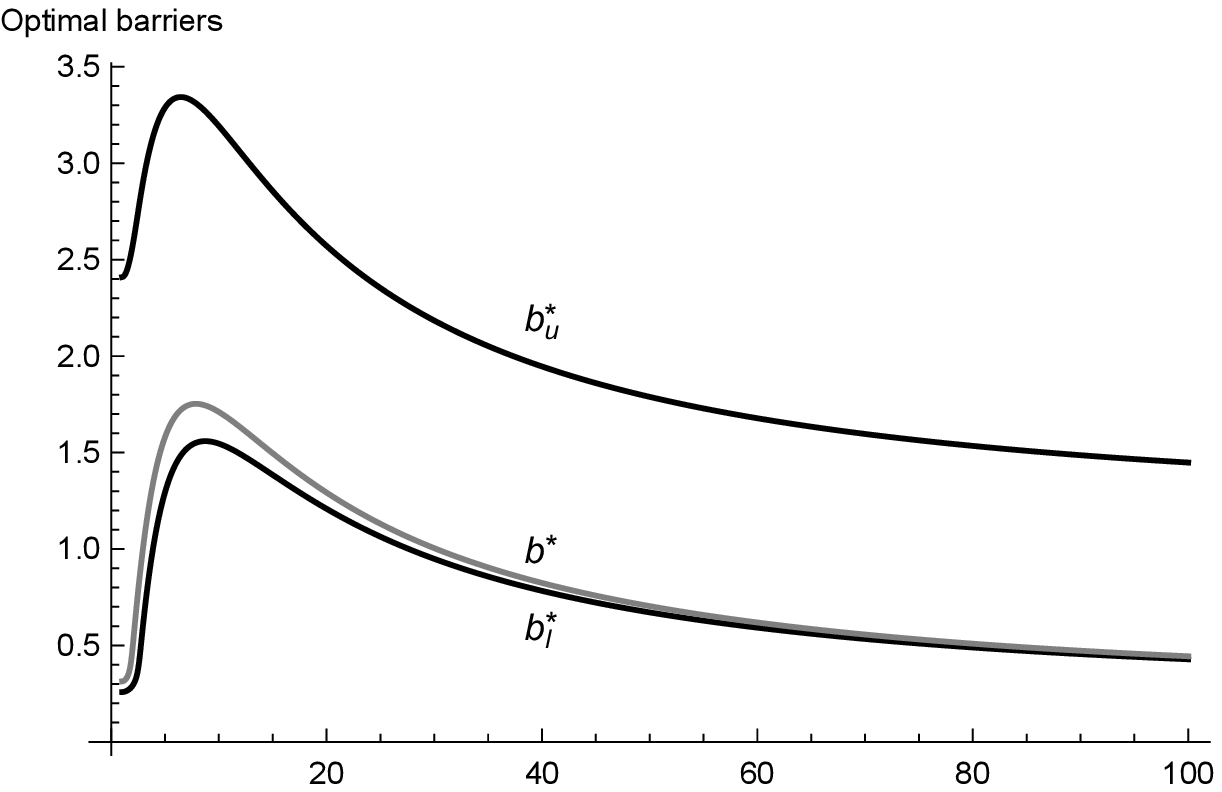}
		\subcaption[fourth caption.]{$x$-axis: Ratio $M$ between claims expectations, with Probability of small claims fixed}\label{fig.sen.muM}
	\end{minipage}	
	~~\begin{minipage}{0.3\textwidth}
	\centering
	\includegraphics[width=1\textwidth]{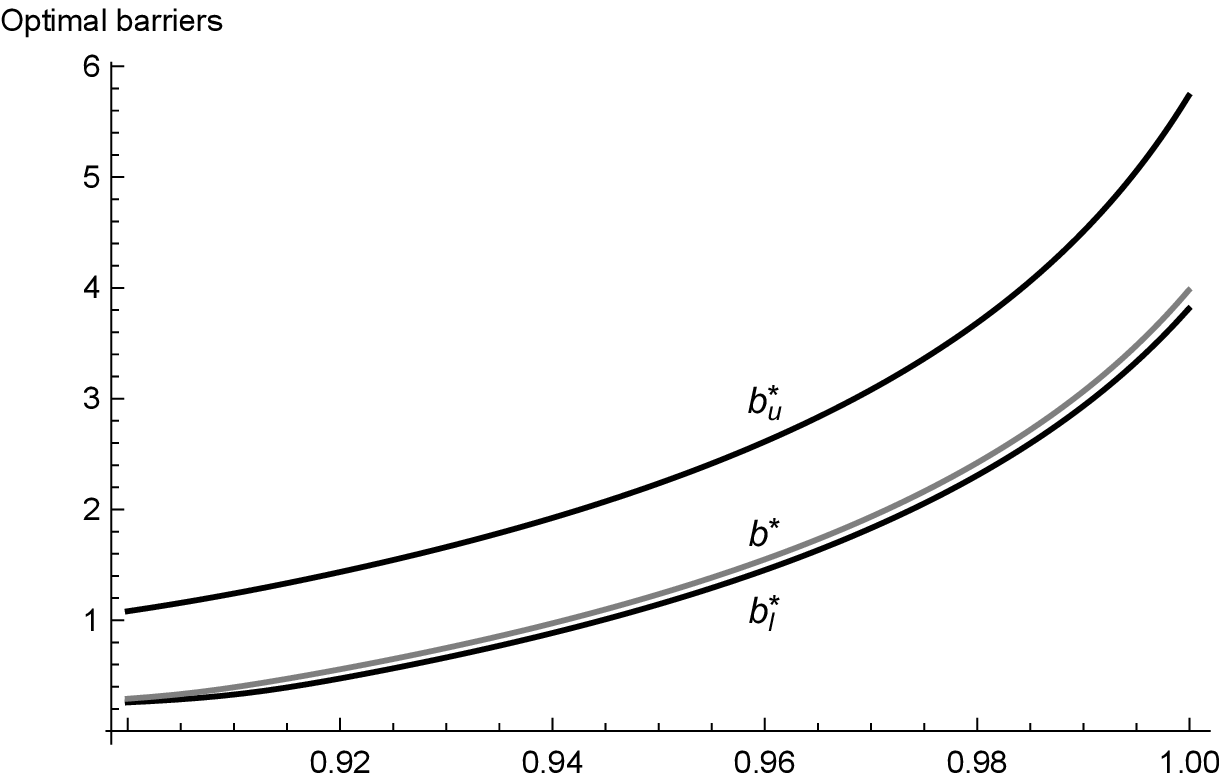}
	\subcaption[fourth caption.]{$x$-axis: Probability of small claims, with expected size of small claims fixed}\label{fig.sen.muE}
\end{minipage}	
	\\
		\begin{minipage}{0.3\textwidth}
		\centering
		\includegraphics[width=1\textwidth]{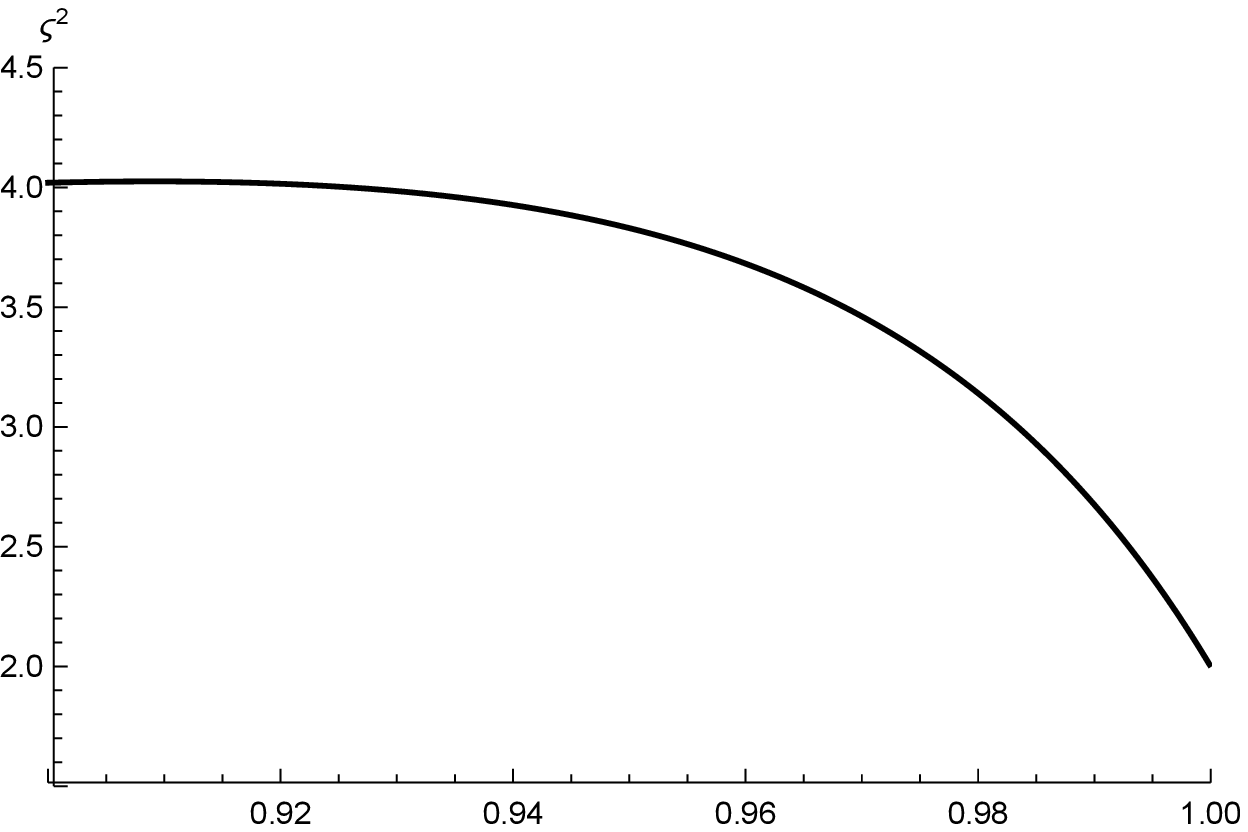}
		\subcaption[first caption.]{$x$-axis: Probability of small claims, with ratio between claims fixed; $y$-axis: overall variability $\varsigma^2$}\label{fig.sen.VarP}
	\end{minipage}%
	~~\begin{minipage}{0.3\textwidth}
		\centering
		\includegraphics[width=1\textwidth]{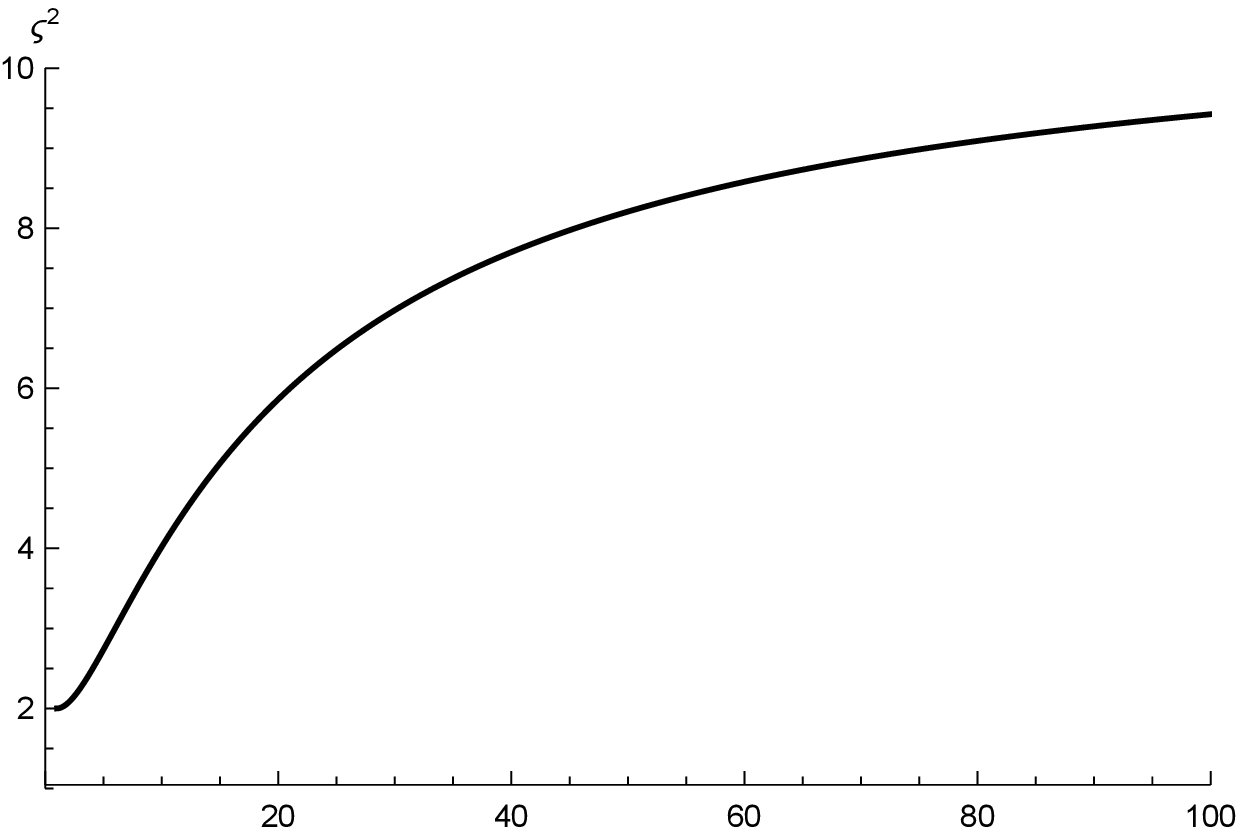}
		\subcaption[second caption.]{$x$-axis: Ratio between claims $M$, with Probability of small claims fixed; $y$-axis: overall variability $\varsigma^2$}\label{fig.sen.VarM}
	\end{minipage}%
	~~\begin{minipage}{0.3\textwidth}
		\centering
		\includegraphics[width=1\textwidth]{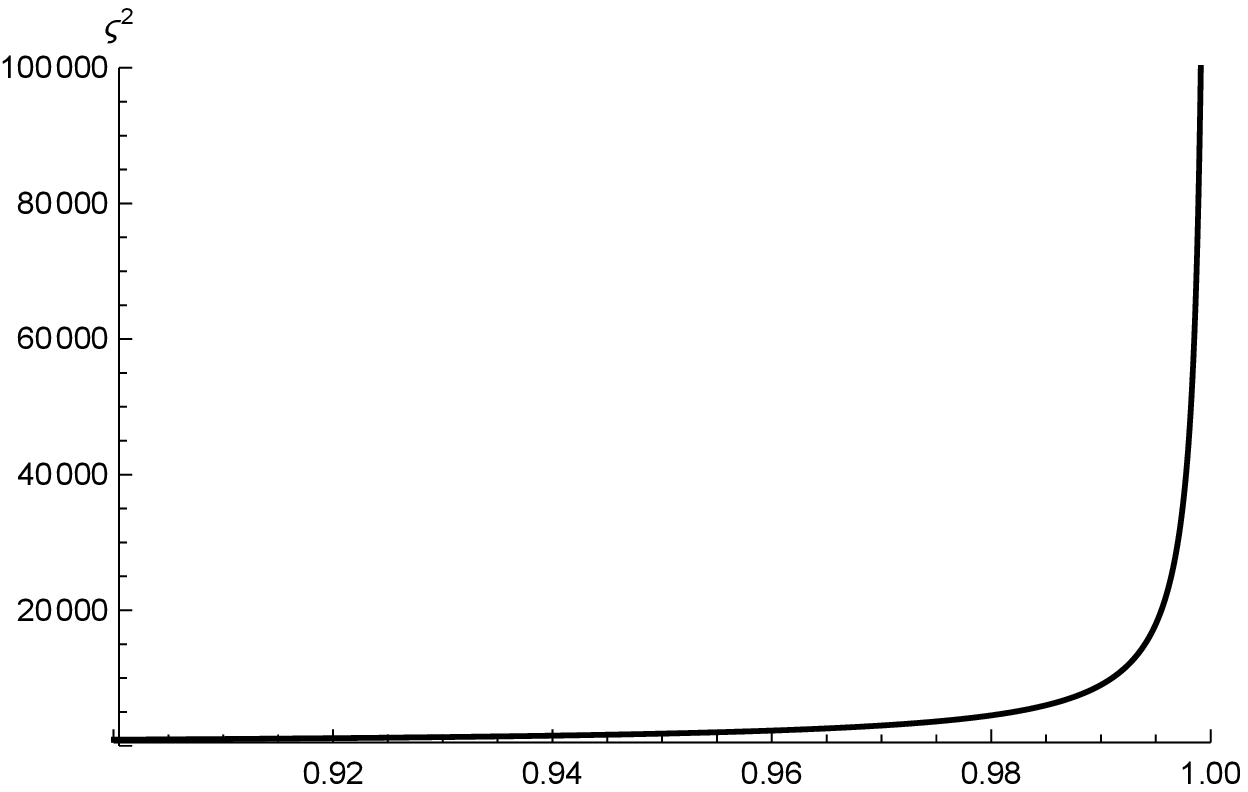}
		\subcaption[fourth caption.]{$x$-axis: Probability of small claims, with expected size of small claims fixed; $y$-axis: overall variability $\varsigma^2$}\label{fig.sen.VarE}
	\end{minipage}	
	
	\caption{Impact of surplus variability on the optimal barriers} \label{fig.sen}
\end{figure}

Figure \ref{fig.sen} shows the impact of the changes of parameters ($\sigma$, $p_1$ and $M$) on the optimal barrier levels ($b_l^*$, $b^*$ and $b_u^*$), where the variability measure $\varsigma^2$ corresponding to the changes are also plotted. 

Figure \ref{fig.sen.sigma} exhibits an increasing then decreasing behaviour. This is because the increased variability would induce cautiousness at first and then deem to be too high for the business to be sustainable, leading eventually to an optimal liquidation at first opportunity. 

For the adjustments of the jumps, we hold the overall expected claim amount per unit time $\mu$  fixed and adjust the parameters accordingly. 
The left column (Figures \ref{fig.sen.muP},\ref{fig.sen.VarP}) adjusts $p_1$, the probabilities of having small jumps (where $\beta_1$ and $\beta_2$, the expected sizes of small and large claims are adjusted accordingly so that their ratio $M$, as well as $\mu$, are fixed). As we see in Figure \ref{fig.sen.VarP}, when the probability of the occurrence of large jumps decreases, the variability of the process decreases. However, barriers in (Figure \ref{fig.sen.muP}) are increasing then decreasing, as a trade-off between occurrence of large jumps (which is decreasing) and their size (which is increasing) operates.

The middle column (Figures \ref{fig.sen.muM},\ref{fig.sen.VarM}) adjusts $M$, the ratio between the expected sizes of the large and small jumps (where $\beta_1$ and $\beta_2$, the expected sizes of small and large claims are adjusted accordingly so that $\mu$ is fixed). In Figure \ref{fig.sen.VarM}, when the ratio between the large and small claims increases, the variability ($\varsigma^2$) of the process increases to a limit. The barriers in Figure \ref{fig.sen.muM} (and beyond) seem to display a convergent behaviour which agrees with Figure \ref{fig.sen.VarM}.

For another comparison, we increase the magnitude of the extreme events while decreases its probability of occurrence in Figures \ref{fig.sen.muE} and \ref{fig.sen.VarE}. To achieve this, in the right column (Figures \ref{fig.sen.muE},\ref{fig.sen.VarE}) we fix the expected value of the small jumps ($1/\beta_1$), decreases the probability of large claims ($p_2$) but at the same time increases its expected value ($1/\beta_2$). We keep the overall expected claims ($\mu$) constant for a fair comparison. It is remarkable that the barriers don't seem to decline to a liquidation-at-first opportunity, even though $\varsigma^2$ becomes very large. Here scarcity of large events seem to overpower the size of the events, even though optimal barriers are still increasing.

\subsection{The impact of the surplus process parameters on the optimal barriers}\label{sec.103}

\begin{figure}[htb]
	\centering

	\begin{minipage}{0.45\textwidth}
		\centering
		\includegraphics[width=1\textwidth]{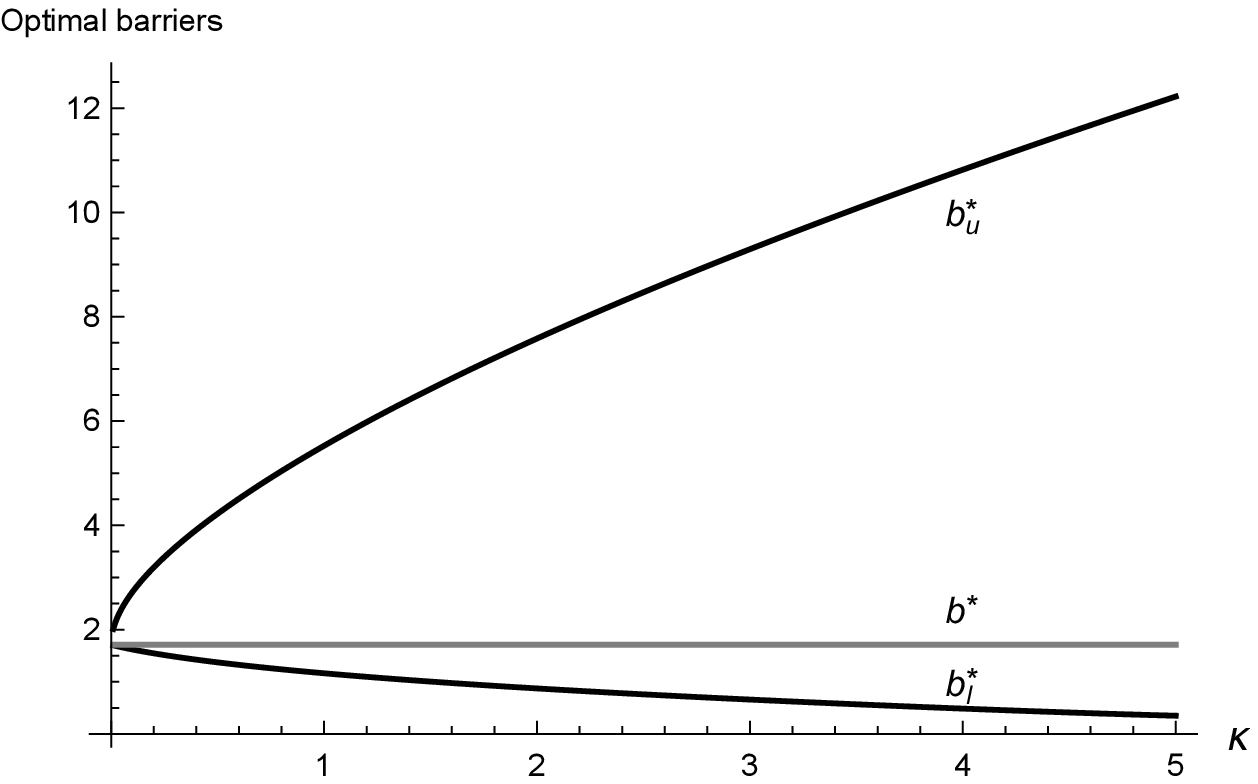}
		\subcaption[first caption.]{$\kappa$}\label{fig.sen.kappa}
	\end{minipage}%
	~~\begin{minipage}{0.45\textwidth}
		\centering
		\includegraphics[width=1\textwidth]{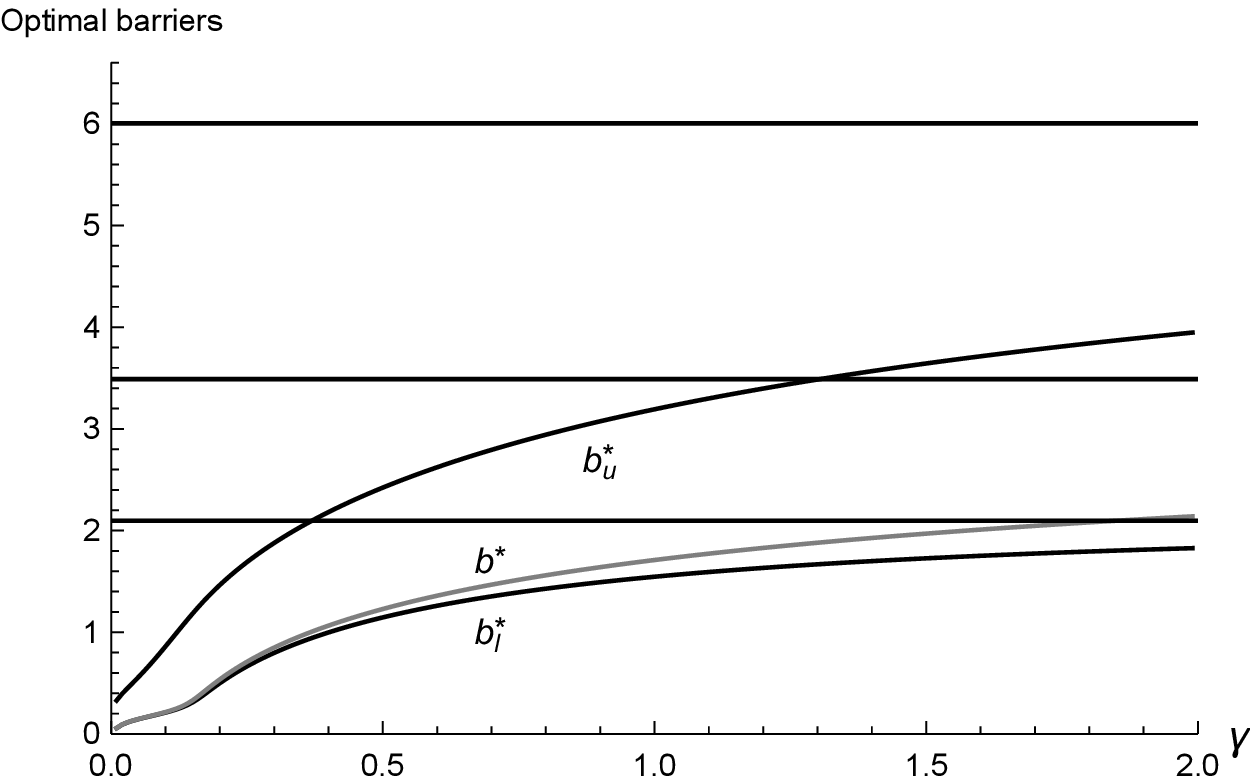}
		\subcaption[second caption.]{$\gamma$}\label{fig.sen.gamma}
	\end{minipage}%
	\\
		\begin{minipage}{0.45\textwidth}
		\centering
		\includegraphics[width=1\textwidth]{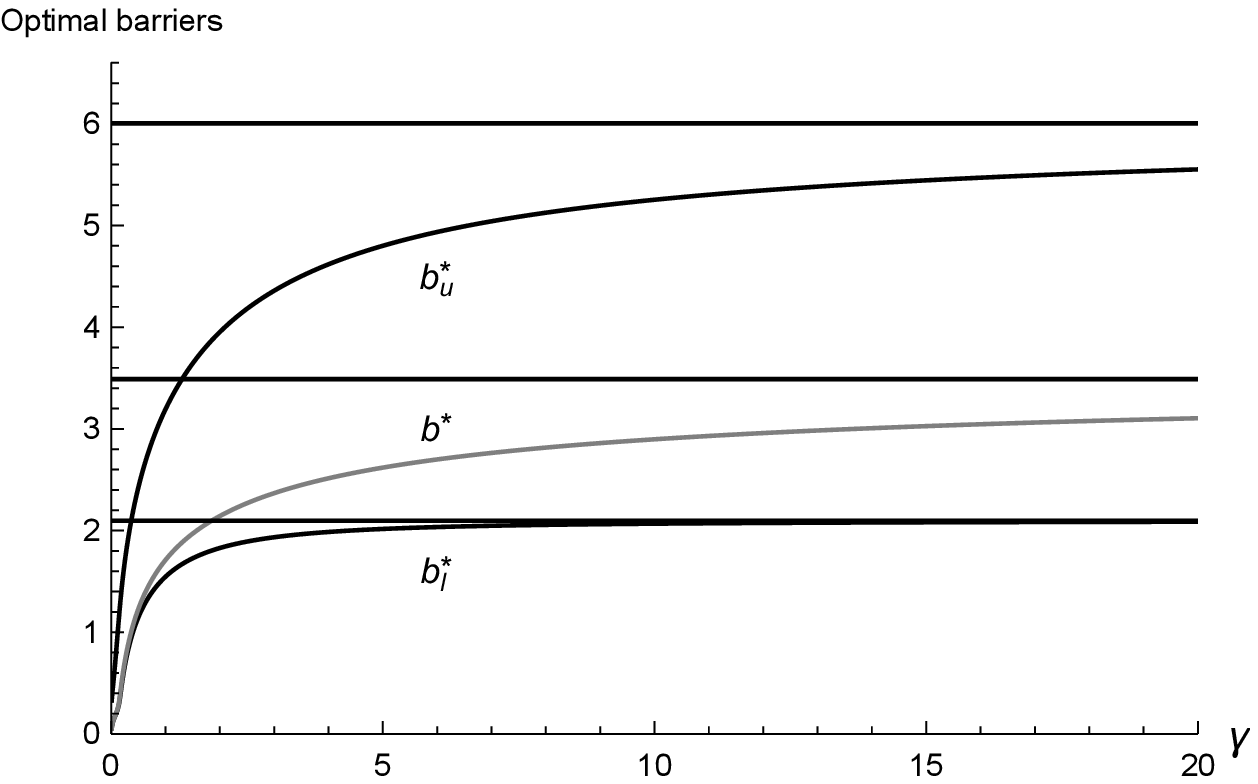}
		\subcaption[second caption.]{$\gamma$ (zoom out)}\label{fig.sen.gamma2}
	\end{minipage}%
	~~\begin{minipage}{0.45\textwidth}
		\centering
		\includegraphics[width=1\textwidth]{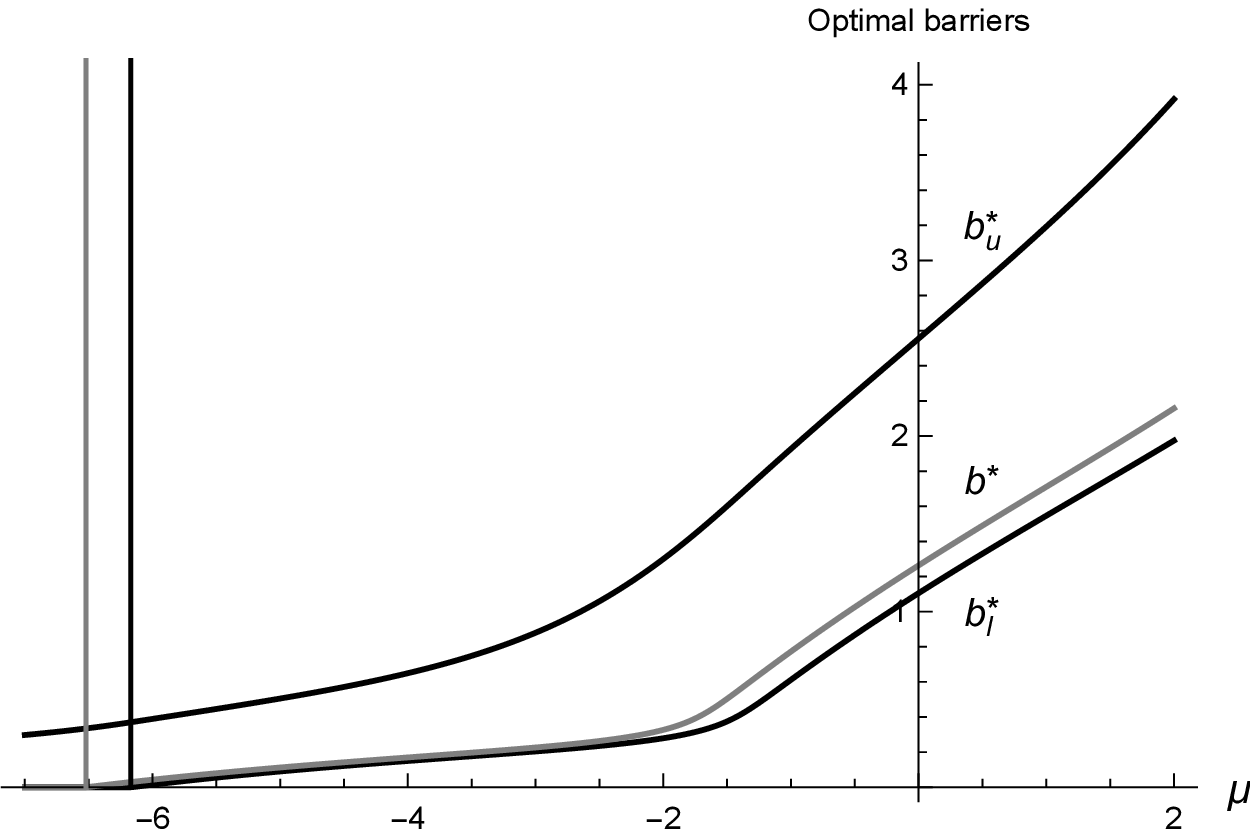}
		\subcaption[second caption.]{$\mu$ (due to change in $c$)}\label{fig.sen.mu}
	\end{minipage}%
	\caption{Sensitivities to parameters (II): Others} \label{fig.sen2}
\end{figure}

Figure \ref{fig.sen2} describes the sensitivities of the barriers to changes in the level of fixed transaction costs, the frequencies of dividend payment opportunities and the premium rate $c$. Figure \ref{fig.sen.kappa} shows that when the fixed transaction costs increase, the optimal periodic barriers $b_u^*$ and $b_l^*$ are moving further away from the periodic barrier $b^*$. This is consistent with the results in the literature \citep*[e.g.][]{BaKyYa13}. 

Figures \ref{fig.sen.gamma}, \ref{fig.sen.gamma2} show the monotonic increase in the optimal barriers with increasing dividend payment frequency $\gamma$. Ultimately they will converge to the barriers when dividends can be paid at any time. It is quite interesting that when $\gamma$ is small, the lower optimal barrier $b_l^*$ behaves similarly to the optimal barrier without the fixed transaction costs. This suggests that the fixed transaction costs are mainly compensated by the upper optimal barrier $b_u^*$, as $b_l$ cannot go below 0, lest a liquidation at first opportunity occurs, which is not optimal in those cases unless $gamma$ is extremely small($<0.002$).

Finally, Figure \ref{fig.sen.mu} plots the change of the optimal barriers corresponding to the changes of the expected gain per unit time $\mu$, solely due to the change of the premium rate $c$. Interestingly, the lower barriers $b_l^*$ and $b^*$ (when there is no transaction costs) are hardly zero unless $\mu$ is too negative. This displays a very different behaviour compared to the case when the surplus process is spectrally positive (i.e. deterministic costs and random gain) where the optimal lower barrier is zero as long as the business is non-profitable indicated by a non-positive $\mu$. This is because in our case, as long as the premium rate $c$ is positive, there is a benefit in having a small but positive surplus as the surplus ``pushes up'' in-between downwards jumps there are chances that no claims arrive before the next jump. On the other hand, once $c$ becomes too small it makes sense to liquidate at first opportunity. This happens on the solid black vertical line. The threshold in absence of fixed transaction costs is different and is illustrated with a solid grey vertical line, which is lower due to absence of transaction costs.

\section{Conclusion}\label{section.conclusion}
In this paper, we determined the form of the optimal periodic dividend strategy when there are fixed transaction costs, when the dividend decisions are Poissonian, and where  the underlying model is a spectrally negative L\'evy process. Extending papers such as \citet*{PeYa16b}, we were able to compute the value function of a periodic $(b_u,b_l)$ strategy concisely in terms of scale functions.

Using an additional assumption that the L\'evy measure has completely monotonic density and imposed the same 2 conditions as \citet*{AvLaWo20d} on the parameters $b_u$ and $b_l$, it was then confirmed that the periodic $({b}_u^*,{b}_l^*)$ (exists and) is optimal.

\section*{Acknowledgments}


This research was  supported under Australian Research Council's Linkage (LP130100723) and Discovery (DP200101859) Projects funding schemes. Hayden Lau acknowledges financial support from an Australian Postgraduate Award and supplementary scholarships provided by the UNSW Australia Business School. The views expressed herein are those of the authors and are not necessarily those of the supporting organisations. 

\section*{References}

\bibliographystyle{elsarticle-harv}

\bibliography{libraries}

\appendix
\section{Value function of a $(b_u,b_l)$ strategy in SNLP}\label{appendix.value.snlp}

We adapt the approach in \citet*{PeYa16b} where the scale of the process is shifted $b_u$ units down, i.e. there are numbers $a<d<-\kappa<0<b$ such that whenever the process $X$ is above or at $0$ at exponential times ($t=e_\gamma$), it jumps to $d$. We denote the difference (the dividend) is denoted as $dL_\gamma^d(t)$ (with the aggregate version being $L_\gamma^d(t)$) and the process is killed upon exiting the interval $[a,b]$.

For $a<0$, we define for any measurable function $f$
\begin{equation}
\Ma{f(x)}:= f(x-a)+\gamma \int_{0}^{x}\Wq(x-y)f(y-a)dy
\end{equation}
and 
\begin{equation}
\Wa(x):=\Ma\Wq(x).
\end{equation}
In particular, we have $\Wa(0)=\Wq(-a)$.
For $a\leq x\leq b$, we have the standard 2 sided exiting identities
\begin{equation}
\Ex(e^{-\delta \tau^+_b};\tau_b^+<\tau_a^-)=\frac{\Wq(x-a)}{\Wq(b-a)},
\end{equation}
\begin{equation}
\Ex(e^{-\delta \tau^-_a};\tau_b^+>\tau_a^-)=\Zq(x-a,\theta)-\Zq(b-a,\theta)\frac{\Wq(x-a)}{\Wq(b-a)}.
\end{equation}
In addition, from equation (2.21), (5.1) and Lemma 5.2 in \cite*{PeYa16b}, we have 
\begin{equation}\label{eqt.2.21}
\Ex(e^{-(\gamma+\delta)\tau_0^-}\Wq(X(\tau_0^-)-a);\tau_0^-<\tau^+_b)= \Wa(x)-\frac{\Wqr(x)}{\Wqr(b)}\Wa(b),
\end{equation}
\begin{equation}\label{eqt.5.1}
\Ex(e^{-\delta e_\gamma};e_\gamma<\tau_b^+\wedge\tau_0^-) = \gamma(\frac{\Wqr(x)}{\Wqr(b)}\Wb(b)-\Wb(x)),
\end{equation}
and
\begin{equation}\label{eqt.lemma5.2}
\Ex(e^{-\delta e_\gamma}X(e_\gamma);e_\gamma<\tau_b^+\wedge\tau_0^-) = \gamma(\frac{\Wqr(x)}{\Wqr(b)}\Wbb(b)-\Wbb(x)),
\end{equation}
where $e_\gamma$ is an independent exponential random variable with mean $1/\gamma$.

By denoting our quantity of interest
\begin{equation}
f_{a,b,d}(x):=\Ex(\int_{0}^{\tau_b^+(\gamma)\wedge\tau_a^-(\gamma)}e^{-\delta t}dL_\gamma^d(t))
\end{equation}
and using the strong Markov property, we have for $x\leq 0$,
\begin{equation*}
f_{a,b,d}(x) = \mathbb{E}_x(e^{-\delta \tau^+_0};\tau_0^+<\tau_a^-)f_{a,b,d}(0)=\frac{\Wq(x-a)}{\Wq(-a)}f_{a,b,d}(0).
\end{equation*}
Hence, for $x\geq 0$, we have by the strong Markov property
\begin{align}
f_{a,b,d}(x) =~& \Ex(e^{-\delta e_\gamma}X(e_\gamma);e_\gamma<\tau_0^-\wedge\tau_b^+)+\Ex(e^{-\delta e_\gamma};e_\gamma<\tau_0^-\wedge\tau_b^+)(f_{a,b,d}(d)-d-\kappa)\nonumber\\
&~+\Ex(e^{-(\gamma+\delta)\tau_0^-}\Wq(X(\tau_0^-)-a);\tau_0^-<\tau_b^+)\frac{f_{a,b,d}(0)}{\Wq(-a)}\nonumber\\
=~& \Ex(e^{-\delta e_\gamma}X(e_\gamma);e_\gamma<\tau_0^-\wedge\tau_b^+)\nonumber\\
&~+\Ex(e^{-\delta e_\gamma};e_\gamma<\tau_0^-\wedge\tau_b^+)(\frac{\Wq(d-a)}{\Wq(-a)}f_{a,b,d}(0)-d-\kappa)\nonumber\\
&~+\Ex(e^{-(\gamma+\delta)\tau_0^-}\Wq(X(\tau_0^-)-a);\tau_0^-<\tau_b^+)\frac{f_{a,b,d}(0)}{\Wq(-a)}.\nonumber
\end{align}
Hence, via equations (\ref{eqt.lemma5.2}), (\ref{eqt.5.1}) and (\ref{eqt.2.21}), we have 
\begin{align}
f_{a,b,d}(x)=~&
\gamma(\frac{\Wqr(x)}{\Wqr(b)}\Wbb(b)-\Wbb(x))\nonumber\\
&~+\gamma(\frac{\Wqr(x)}{\Wqr(b)}\Wb(b)-\Wb(x))\times(\frac{\Wq(d-a)}{\Wq(-a)}f_{a,b,d}(0)-d-\kappa)\nonumber\\
&~+(\Wa(x)-\frac{\Wqr(x)}{\Wqr(b)}\Wa(b))\frac{f_{a,b,d}(0)}{\Wq(-a)}.\nonumber
\end{align}

Since we are only interested in the case when $b\rightarrow\infty$, we should take the limit before calculating $f_{a,b,d}(0)$. By using (\ref{eqt.lim1}), (\ref{eqt.lim2}) and (\ref{eqt.lim3}), we get
\begin{align}\label{eqt.fad}
f_{a,d}(x):=~&\lim_{b\rightarrow\infty}f_{a,b,d}(x)\nonumber\\
=~&\gamma(\frac{1}{\phiqr^2}\Wqr(x)-\Wbb(x))\nonumber\\
&~+\gamma(\frac{1}{\phiqr}\Wqr(x)-\Wb(x))\times(\frac{\Wq(d-a)}{\Wq(-a)}f_{a,d}(0)-d-\kappa)\nonumber\\
&~+(\Wa(x)-\Wqr(x)\Zqr(-a))\frac{f_{a,d}(0)}{\Wq(-a)}\nonumber\\
=~&\gamma(\frac{1}{\phiqr^2}\Wqr(x)-\Wbb(x))+\gamma(-d-\kappa)(\frac{1}{\phiqr}\Wqr(x)-\Wb(x))\nonumber\\
&~+\frac{f_{a,d}(0)}{\Wq(-a)}\Big\{r\Wq(d-a)(\frac{1}{\phiqr}\Wqr(x)-\Wb(x))+\Wa(x)-\Wqr(x)\Zqr(-a)\Big\}.
\end{align}

When $X$ is of bounded variation, $\Wqr(0)>0$, and hence we put $x=0$ in equation (\ref{eqt.fad}) to obtain
\begin{align}
f_{a,d}(0)=~& \gamma\frac{1}{\phiqr^2}\Wqr(0)+\gamma(-d-\kappa)\frac{1}{\phiqr}\Wqr(0)\nonumber\\
&~+\frac{f_{a,d}(0)}{\Wq(-a)}\Big\{\gamma\Wq(d-a)\frac{1}{\phiqr}\Wqr(0)+\Wq(-a)-\Wqr(0)\Zqr(-a)\Big\}\nonumber\\
=~&\frac{\Wqr(0)}{\phiqr}(\gamma(\frac{1}{\phiqr}-d-\kappa))+\frac{\Wqr(0)}{\phiqr}\frac{f_{a,d}(0)}{\Wq(-a)}\Big\{\gamma\Wq(d-a)-\phi\Zqr(-a)\Big\}+f_{a,d}(0),\nonumber
\end{align}
or
\begin{equation}\label{eqt.f0}
\frac{f_{a,d}(0)}{\Wq(-a)}=\frac{\gamma(\frac{1}{\phiqr}-d-\kappa)}{\phiqr\Zqr(-a)-\gamma\Wq(d-a)}.
\end{equation}
When $X$ is of unbounded variation, by denoting the event
\begin{equation}
E_B:=\{e_\gamma<\tau_0^-\}\cup \{\tau_b^+<\zeta\}\cup\{\tau_a^-<\zeta\},
\end{equation}
where $\zeta$ is the lifetime of an excursion away from $0$, $a\leq 0\leq b$. In addition, we denote $\T$ the first time an excursion in the event $E_B$ occurs and $l_\T$ the starting point of the excursion, i.e.
\begin{equation}
l_\T:=\sup\{t<\T:X(t)=0\}.
\end{equation}

From equation (6.6) in \citet*{PeYa16b}, by denoting $\bar{T}^-_0=l_{T_{E_B}}+\tau^-_0\circ \vartheta_{l_{T_{E_B}}}$, where $\vartheta$ is the shifting operator, we have 
\begin{equation}\label{eqt.6.6.1}
{\mathbb{E}}(e^{-\delta(l_\T+e_\gamma)}X(l_\T+e_\gamma);l_\T+e_\gamma<\bar{T}_0^-\wedge\tau_b^+) = \gamma\frac{\Wq(-a)}{\Wa(b)}\Wbb(b),
\end{equation}
\begin{equation}\label{eqt.6.6.2}
{\mathbb{E}}(e^{-\delta(l_\T+e_\gamma)};l_\T+e_\gamma<\bar{T}_0^-\wedge\tau_b^+) = \gamma\frac{\Wq(-a)}{\Wa(b)}\Wb(b).
\end{equation}

Regarding the limiting behaviour when $b\rightarrow\infty$, we have
\begin{equation}\label{eqt.lim1}
\lim_{b\rightarrow\infty}\frac{\Wbb(b)}{\Wqr(b)}=\frac{1}{\phiqr^2},
\end{equation}
\begin{equation}\label{eqt.lim2}
\lim_{b\rightarrow\infty}\frac{\Wb(b)}{\Wqr(b)}=\frac{1}{\phiqr},
\end{equation}
and
\begin{equation}\label{eqt.lim3}
\lim_{b\rightarrow\infty}\frac{\Wa(b)}{\Wqr(b)}=\Zqr(-a).
\end{equation}

On the other hand, when $X$ is of unbounded variation, we proceed as in section 6 in \cite*{PeYa16b} to yield
\begin{align}
f_{a,b,d}(0)=~&\mathbb{E}(e^{-\delta(l_\T+e_\gamma)}X(l_\T+e_\gamma);l_\T+e_\gamma<\bar{T}_0^-\wedge\tau_b^+)\nonumber\\
&~+\mathbb{E}(e^{-\delta(l_\T+e_\gamma)};l_\T+e_\gamma<\bar{T}_0^-\wedge\tau_b^+) \times (f_{a,b,d}(d)-d-\kappa)\nonumber\\
=~&\gamma\frac{\Wq(-a)}{\Wa(b)}\Wbb(b)+\gamma\frac{\Wq(-a)}{\Wa(b)}\Wb(b)(f_{a,b,d}(d)-d-\kappa).\nonumber
\end{align}
By passing the limit $b\rightarrow\infty$, we get
\begin{align}
f_{a,d}(0) =~& \frac{\gamma\Wq(-a)}{\phi^2 \Zqr(-a) }+\frac{\gamma\Wq(-a)}{\phi\Zqr(-a)}(\frac{\Wq(d-a)}{\Wq(-a)}f_{a,b,d}(0)-d-\kappa)\nonumber\\
=~&  \frac{ \gamma\Wq(-a)}{\phiqr \Zqr(-a) }(\frac{1}{\phiqr}-d-\kappa)+\frac{\gamma\Wq(d-a)}{\phiqr\Zqr(-a)}f_{a,b,d}(0),\nonumber
\end{align}
or
\begin{align*}
\phiqr\Zqr(-a)f_{a,d}(0) =~& \gamma\Wq(-a)(\frac{1}{\phiqr}-d-\kappa)+\gamma\Wq(d-a)f_{a,d}(0),
\end{align*}
which yields
\begin{equation*}
\frac{f_{a,d}(0)}{\Wq(-a)}=\frac{\gamma(\frac{1}{\phiqr}-d-\kappa)}{\phiqr\Zqr(-a)-\gamma\Wq(d-a)},
\end{equation*}
the same as (\ref{eqt.f0}).

Thus, plugging (\ref{eqt.f0}) back to (\ref{eqt.fad}), we obtain
\begin{align}\label{value.kappa.progress}
f_{a,d}(x)=~&\gamma(\frac{1}{\phiqr^2}\Wqr(x)-\Wbb(x))+\gamma(-d-\kappa)(\frac{1}{\phiqr}\Wqr(x)-\Wb(x))\nonumber\\
&~+\frac{f_{a,d}(0)}{\Wq(-a)}\Big\{\gamma\Wq(d-a)(\frac{1}{\phiqr}\Wqr(x)-\Wb(x))+\Wa(x)-\Wqr(x)\Zqr(-a)\Big\}\nonumber\\
=~&\frac{\Wqr(x)}{\phiqr}\gamma(\frac{1}{\phiqr}-d)-\gamma(\Wbb(x)-d \Wb(x)-\kappa\Wb(x))\nonumber\\
&~+\frac{f_{a,d}(0)}{\Wq(-a)}\Big\{\Wa(x)-\gamma\Wq(d-a)\Wb(x)\Big\}\nonumber\\
&~+\frac{f_{a,d}(0)}{\Wq(-a)}\Big\{\gamma\Wq(d-a)-\phiqr\Zqr(-a)\Big\}\frac{1}{\phiqr}\Wqr(x)\nonumber\\
=~&\frac{\Wqr(x)}{\phiqr}\gamma(\frac{1}{\phiqr}-d-\kappa)-\gamma(\Wbb(x)-d \Wb(x)-\kappa\Wb(x))\nonumber\\
&~+\frac{f_{a,d}(0)}{\Wq(-a)}\Big\{\Wa(x)-\gamma\Wq(d-a)\Wb(x)\Big\}\nonumber\\
&~-\gamma(\frac{1}{\phiqr}-d-\kappa)\frac{\Wqr(x)}{\phiqr},
\end{align}
or
\begin{align}
f_{a,d}(x)=~&\frac{\gamma(\frac{1}{\phiqr}-d-\kappa)}{\phiqr\Zqr(-a,)-\gamma\Wq(d-a)}\Big(\Wa(x)-\gamma\Wq(d-a)\Wb(x)\Big)\nonumber\\&-\gamma\Big(\Wbb(x)-d \Wb(x)-\kappa\Wb(x)\Big).\label{value.kappa}
\end{align}

Next, the smoothness condition ((\ref{smooth.condition}) shifted $b_u$ units downward)
\begin{equation}
f_{a,d}(0)=f_{a,d}(d)-d-\kappa
\end{equation}
can be rewriten as
\begin{equation}
-d-\kappa = \frac{f_{a,d}(0)}{\Wq(-a)}(\Wq(-a)- \Wq(d-a)).
\end{equation}
Hence, in view of (\ref{value.kappa.progress}), we have
\begin{align}\label{value.kappa.progress2}
&f_{a,d}(x)\nonumber\\ =~& \frac{f_{a,d}(0)}{\Wq(-a)}\Big(\Wa(x)-\gamma\Wq(d-a)\Wb(x)\Big)-\gamma\Big(\Wbb(x)-d \Wb(x)-\kappa\Wb(x)\Big)\nonumber\\
=~&\frac{f_{a,d}(0)}{\Wq(-a)}\Big(\Wa(x)-\gamma\Wq(d-a)\Wb(x)\Big)-\gamma(-d-\kappa)\Wb(x)-\gamma\Wbb(x)\nonumber\\
=~&\frac{f_{a,d}(0)}{\Wq(-a)}\Big(\Wa(x)-\gamma\Wq(d-a)\Wb(x)\Big)-\gamma\frac{f_{a,d}(0)}{\Wq(-a)}(\Wq(a)- \Wq(d-a))\Wb(x)-\gamma\Wbb(x)\nonumber\\
=~&\frac{f_{a,d}(0)}{\Wq(-a)}\Big(\Wa(x)-\gamma\Wq(d-a)\Wb(x)+\gamma\Wq(d-a)\Wb(x)-\gamma\Wq(-a)\Wb(x)\Big)-\gamma\Wbb(x)\nonumber\\
=~&\frac{f_{a,d}(0)}{\Wq(-a)}\Big(\Wa(x)-\gamma\Wq(-a)\Wb(x)\Big)-\gamma\Wbb(x).
\end{align}
Now, plugging in $x=0$ in (\ref{value.kappa}), we get
\begin{equation*}
f_{a,d}(0) = \frac{\gamma(\frac{1}{\phi}-d-\kappa)}{\phiqr\Zqr(-a)-\gamma\Wq(d-a)}\Wq(-a),
\end{equation*} 
or
\begin{align*}
\frac{f_{a,d}(0)}{\Wq(-a)}=~&\frac{\frac{\gamma}{\phiqr}+\gamma\frac{f_{a,d}(0)}{\Wq(-a)}(\Wq(-a)- \Wq(d-a))}{\phiqr\Zqr(-a)-\gamma\Wq(d-a)},
\end{align*}
which results in 
\begin{equation}
\frac{f_{a,d}(0)}{\Wq(-a)} = \frac{\gamma}{\phiqr}\frac{1}{\phiqr\Zqr(-a)-\gamma\Wq(-a)}.
\end{equation}
Plugging this back to (\ref{value.kappa.progress2}), we retrive the value function of a periodic barrier strategy (at barrier level $a$), which appears in \cite*{PeYa16b}.

Finally, we shall perform a horizontal transformation of the axis such that we have $0<\kappa<g = b_u-b_l$, $0<b_u,b_l$, the process is ruin when downcrossing $0$ and whenever at Poissonian times the process is above $b_u$ it jumps to $b_l$. This gives the value function of a periodic $(b_u,b_l)$ strategy as desired.

\end{document}